\documentclass[a4paper,11pt]{article}

\usepackage{graphicx,a4wide,listings}
\usepackage[intlimits]{amsmath}
\usepackage{amsxtra, amssymb,fancyhdr, amsthm,latexsym}
\usepackage{amsmath}
\usepackage[alphabetic,nobysame]{amsrefs}
\usepackage{amssymb, color}

\usepackage[english]{babel}
\usepackage[ansinew]{inputenc}
\usepackage[bookmarks]{}
\usepackage{enumerate}
\numberwithin{equation}{section}
\numberwithin{figure}{section}
\usepackage[colorlinks=true]{hyperref}
\hypersetup{urlcolor=blue, citecolor=red}
\usepackage{enumitem}
\usepackage{tikz}
\usepackage{caption}
\usepackage{subfigure}

\newtheorem{Theorem}{Theorem}[section]
\newtheorem{Corollary}[Theorem]{Corollary}

\newtheorem{Proposition}[Theorem]{Proposition}

\newtheorem{rmk}[Theorem]{Remark}

\newtheorem{Definition}{Definition}[section]

\newcommand{\nn}{\nonumber}
\newcommand{\R}{{\mathbb R}}

\newcommand{\N}{{\mathbb N}}
\newcommand{\Rd}{\R^d}

\newcommand{\Grad}{\nabla_{\!x}}
\newcommand{\Gradx}{\nabla_{\!x}}
\newcommand{\del}{\partial}
\newcommand{\dr}{ \, {\rm d} r}
\newcommand{\dx}{ \, {\rm d} x}
\newcommand{\dy}{ \, {\rm d} y}

\newcommand{\dt}{ \, {\rm d} t}
\newcommand{\dv}{ \, {\rm d} v}

\newcommand{\ds}{\, {\rm d} s}

\newcommand{\dtau}{\, {\rm d} \tau}

\newcommand{\One}{\boldsymbol{1}}

\newcommand{\Denote}{\stackrel{\triangle}{=}}

\newcommand{\CalA}{{\mathcal{A}}}

\newcommand{\CalF}{{\mathcal{F}}}
\newcommand{\setF}{{\mathcal{S}}}

\newcommand{\CalM}{{\mathcal{M}}}

\newcommand{\CalP}{{\mathcal{P}}}
\newcommand{\Pc}{{\mathcal{P}_c}}
\newcommand{\CalR}{{\mathcal{R}}}

\newcommand{\Gradv}{\nabla_{\!v}}
\newcommand{\Eps}{\epsilon}

\newcommand{\Epsf}{f_\Eps}

\newcommand{\Epsrho}{\rho_\Eps}
\newcommand{\Epsu}{u_\Eps}

\newcommand{\Epsx}{x_\Eps}
\newcommand{\Epsv}{v_\Eps}
\newcommand{\EpsU}{D_\Eps}
\newcommand{\etaEps}{\xi_\Eps}
\newcommand{\tetaEps}{\tilde{\xi}_\Eps}

\newcommand{\Supp}{\text{supp}}

\newcommand{\Suppxeps}{S_\epsilon}
\newcommand{\Suppveps}{V_\epsilon}
\newcommand{\Supprho}{S_\rho}

\newcommand{\Vmax}{v_{M}}
\newcommand{\Gfluc}{I_\Eps}
\newcommand{\Rem}{G_\Eps} 
\newcommand{\Fpmap}{\mathcal{G}}

\newcommand{\rhonum}{\vec{\rho}}
\newcommand{\unum}{\vec{u}}
\newcommand{\bnum}{\vec{b}}
\newcommand{\enum}{\vec{e}}

\newcommand{\abs}[1]{\left\lvert#1\right\rvert}
\newcommand{\norm}[1]{\left\lVert#1\right\rVert}

\newcommand{\vint}[1]{\langle #1 \rangle}
\newcommand{\vpran}[1]{\left( #1 \right)}

\date{\today}
\begin{document}

\title{First-order aggregation models with alignment}

\author{R. C. Fetecau  \thanks{Department of Mathematics, Simon Fraser University, 8888 University Dr., Burnaby, BC V5A 1S6, Canada. Email: van@math.sfu.ca}
\and W. Sun \thanks{Department of Mathematics, Simon Fraser University, 8888 University Dr., Burnaby, BC V5A 1S6, Canada. Email: weirans@math.sfu.ca}
\and Changhui Tan \thanks{Department of Mathematics \& CSCAMM, University of Maryland, College Park, MD, 20742, U.S.A. Email: ctan@cscamm.umd.edu}
}
\maketitle

\begin{abstract}
We include alignment interactions in a well-studied first-order attractive-repulsive macroscopic model for aggregation. The distinctive feature of the extended model is that the equation that specifies the velocity in terms of the population density, becomes {\em implicit}, and can have non-unique solutions. 
We investigate the well-posedness of the model and show rigorously how it can be obtained as a macroscopic limit of a second-order kinetic equation.  We work within the space of probability measures with compact support and use mass transportation ideas and the characteristic method as essential tools in the analysis. A discretization procedure that parallels the analysis is formulated and implemented numerically in one and two dimensions.
\medskip

\end{abstract}


\section{Introduction}
\label{sect:intro}

The literature on self-organizing behaviour or swarming has grown dramatically over the last years. A variety of mathematical models has been proposed, which have origin in biological applications (e.g., self-collective behaviour seen in species such as fish, birds or insects \cite{Camazine_etal}), as well as in social sciences and engineering (e.g., opinion formation  \cite{MotschTadmor2014}, social networks \cite{Jackson2010}, robotics and space missions \cite{JiEgerstedt2007}). The main aspect is the modelling of the social interactions between the members of a group; due to such inter-individual interactions, self-organization may occur in a physical space (insect swarms, fish schools, robots) or, more abstractly, in an opinion space.

One approach in modelling aggregation is to consider individuals/organisms as point particles and design either an ordinary differential equation (ODE) or a discrete-time equation to model their evolution. Another is to formulate a partial differential equation (PDE) that governs the time evolution of the population density field. These two approaches result from the various descriptions that one can take in modelling aggregation behaviour: particle-based/microscopic or continuum/macroscopic.
We refer to \cite{CarrilloVecil2010} for a recent review of aggregation models, where in particular, it is shown how microscopic models can be related to macroscopic ones via kinetic theory. 

Three types of social interactions have been commonly considered in the literature on mathematical aggregations: attraction, repulsion, and alignment. Consequently, aggregation models can be distinguished in terms of which of these interactions are being accounted for.  Some models consider only a subset of these interactions (just attraction and repulsion \cite{M&K,BertozziCarilloLaurent,D'Orsogna_etal,BT2011} or just alignment \cite{CS2007,CFRT2010}), others account for all three of them. Models of the latter type are typically referred to as ``three-zone'' models, as each particular interaction type acts at different ranges (repulsion acts at short distances, while alignment and attraction are present at intermediate and long ranges, respectively). This class of models has had many successful applications in biological and sociological modelling  \cite{Couzin_etal, Reynolds1987,Miller_etal2012}. 

Aggregation models  (discrete or continuous) may also differ in how the velocity field is determined. There are second-order/dynamic models, typically in the form of the Newton's second law, where a differential equation for the evolution of the velocity is being provided \cite{D'Orsogna_etal, CS2007},  and first-order/kinematic models where the velocity is prescribed in terms of the spatial configuration \cite{M&K, Topaz:Bertozzi, BertozziCarilloLaurent}.  The aim of the present paper is to extend, by including alignment interactions, a first-order continuum model for aggregation that attracted a high amount of interest in recent literature  \cite{M&K,BodnarVelasquez2, BertozziLaurent, KoSuUmBe2011, LeToBe2009, FeHuKo11}. Below we  introduce the extended model and its derivation, then point out the fundamental issues that arise with such an extension, and how we address these issues in the present paper.

Consider the following continuum model for the evolution of the macroscopic density function $\rho(t,x)$ in $\Rd$: 
\begin{subequations}
\label{eq:macroscopic}
\begin{gather}
    \del_t \rho \, + \Grad \cdot (\rho u) = 0, 
\qquad \rho \big|_{t=0} = \rho_0(x), \label{eq:rho}
\\
    \Phi(t,x) u(t, x) = \int_{\R^{d}} \phi(|x-y|) \rho(t,y) u(t,y) \dy 
                             - \Grad K \ast \rho(t,x),
\label{eq:u-implicit}
\end{gather} 
\end{subequations}
where $\phi$ is an {\em influence} function that controls the alignment interactions, $\Phi = \phi \ast \rho$, and $K$ is an attractive-repulsive interaction potential. The asterisk denotes spatial convolution. Hence, the model consists in an active transport equation for the density $\rho$, with velocity field $u$ defined by \eqref{eq:u-implicit}. The coefficient $\Phi(t,x)$ in the left-hand-side of \eqref{eq:u-implicit} has the interpretation of the total influence received at location $x$ and time $t$ from the rest of the group. The right-hand-side of  \eqref{eq:u-implicit} has two terms: the first models alignment, and the second models attraction and repulsion, as detailed below.

Alignment is modelled through an averaging mechanism that allows individuals to adjust their velocities relative to the velocities of the others. Specifically, the velocity at location $x$ and time $t$ is assumed to depend non-locally on the velocities $u(t,y)$ at locations $y$ within the alignment interaction range set by the support of the influence function $\phi$. The first-term in the right-hand-side of \eqref{eq:u-implicit} captures this averaging process, with weights/interaction strengths given by $\phi(|x-y|)$, assumed to depend only on the relative distance between locations $x$ and $y$.

Attraction and repulsion are modelled by the convolution of the gradient of the interaction potential with the population density. In brief, individuals are assumed to  repel each other at short ranges, to create a comfort zone around them, but attract each other once they distance themselves too far apart. Equation \eqref{eq:rho} with the velocity field given {\em solely} by this term, i.e., $u(t,x) = -\nabla_x K \ast \rho(t,x)$, constitutes the aggregation model that has been referred to above and which the present paper generalizes.  A variety of issues has been investigated during the last decade for this explicit attractive-repulsive aggregation model, including the well-posedness of solutions \cite{BodnarVelasquez2, BertozziLaurent, BertozziLaurentRosado,CaFrFiLaSl2011}, the long-time behaviour of solutions \cite{M&K,Burger:DiFrancesco, KoSuUmBe2011, LeToBe2009, BT2011, FeHuKo11}, and the derivation of the continuum model as a mean-field limit \cite{CarrilloChoiHaurayCISM}. 

We point out an essential feature of equation \eqref{eq:u-implicit}, which is that it is an {\em implicit} equation in $u$, and can have {\em non-unique} solutions (e.g., due to translational invariance, one can add an arbitrary function of $t$ to any solution of \eqref{eq:u-implicit}, and obtain a different solution). This is a key challenge brought up by the inclusion of alignment interactions in the explicit aggregation model from \cite{M&K}. Addressing this challenge is one of the major goals of the present paper. The non-uniqueness of solutions to models of type \eqref{eq:macroscopic} has been noted
 in \cite{Miller_etal2012}, but no resolution was offered.  To our best knowledge, the present paper is the first systematic study of a {\em first-order} continuum model for aggregation that includes {\em both} attractive/repulsive and alignment interactions. 

The origin of the macroscopic model \eqref{eq:macroscopic} can be traced back to the following second-order discrete model derived from Newton's second law. Suppose there are $N$ particles in $\Rd$, whose positions and velocities  denoted by $x_i$ and $v_i$, respectively ($i=1,\dots,N$), evolve according to the following system of ODE's:
\begin{subequations}
\label{eqn:so-discrete}
\begin{align}
\frac{dx_i}{dt}&=v_i, \label{eqn:so-dxidt} \\
\Eps\frac{dv_i}{dt}&=\frac{1}{N}
\sum_{j\neq i}\phi(|x_j-x_i|)(v_j-v_i)-\frac{1}{N}\sum_{j\neq i}\nabla_{x_i} K(x_i-x_j), \label{eqn:so-dvidt}
\end{align}
\end{subequations}
for $i=1,\dots,N$. Here, it has been assumed that that all particles have the same mass $m_i=\Eps$. The functions $\phi$ and $K$ have similar meaning as in \eqref{eq:macroscopic}.

Without the attractive-repulsive term modelled by the second term in the right-hand-side of \eqref{eqn:so-dvidt}, model \eqref{eqn:so-discrete}  represents the celebrated model of Cucker and Smale \cite{CS2007}. It is well-known that for certain influence functions $\phi$, the Cucker-Smale model successfully captures the unconditional \emph{flocking} phenomenon, where individuals align their velocities into a certain asymptotic direction \cite{HT2008, HL2009, CFRT2010}. Comprehensive literature also exists on second-order attractive-repulsive models without alignment \cite{D'Orsogna_etal, Chuang_etal, CarrilloVecil2010}. Second-order models with {\em both} alignment and attraction/repulsion have also been studied  \cite{CD2010, LiLuKe2008, CCR2011, AgIlRi2011}, though in not as much depth and detail as the models with the two sets of forces considered separately.

The passage from the second-order discrete model  \eqref{eqn:so-discrete} to the first-order macroscopic model \eqref{eq:macroscopic} is done in two steps. First, for each fixed $\Eps > 0$, one can take the limit $N\to \infty$ in  \eqref{eqn:so-discrete} and reach,  by BBGKY hierarchies or mean field limits (see e.g., \cite{HT2008} and the review in \cite{CarrilloVecil2010}), the following kinetic equation for the density $f(t,x,v)$ at position $x \in \Rd$ and velocity $v \in \Rd$:
\begin{subequations}
\label{eq:kinetic}
\begin{gather} 
   \del_t \Epsf + v \cdot \Grad \Epsf 
   = \frac{1}{\Eps} \Gradv \cdot \vpran{F[\Epsf] \Epsf},
\qquad 
   \Epsf \big|_{t=0} = f_0(x, v), \label{eq:f-Eps} 
   \\
  F[\Epsf] 
    = \int_{\R^{2d}} \phi(|x-y|) (v - v^\ast) \Epsf(t, y, v^\ast) \dy\dv^\ast
       +\int_{\R^{2d}}\Gradx K(x-y) \Epsf(t,y,v^\ast)dydv^\ast. \label{eq:FfEps}  
\end{gather}
\end{subequations}
Rigorous derivations of mean field limits starting from particle systems is a classical subject, comprising an extensive body of  works. Some of the most recent works include the mean field limit of the Cucker-Smale model \cite{HL2009, CFRT2010}, as well as extensions of these results to include general aggregation models of the form \eqref{eqn:so-discrete} \cite{CCR2011}. 

The second step in passing from \eqref{eqn:so-discrete} to \eqref{eq:macroscopic} is to send $\Eps \to 0$ in the kinetic equation \eqref{eq:kinetic} and derive \eqref{eq:macroscopic} as a hydrodynamic limit. The rigorous treatment of this limit constitutes in fact a major component of the present work. Passage from kinetic to macroscopic equations is also a vast topic, extensively studied for instance in the context of hydrodynamic limits of the nonlinear Botlzmann equations. It is beyond the scope of this introduction to give a detailed account of this well-established research area, we simply refer here to a recent review paper \cite{SR2014} and the references therein.

As indicated above, a major issue that arises when one considers the first-order model \eqref{eq:macroscopic} is the {\em non-uniqueness} of solutions to \eqref{eq:u-implicit}. In the present paper we resolve this non-uniqueness issue for the case when the interaction potential is {\em symmetric} about the origin, i.e., it satisfies
\begin{equation}
\label{eqn:Ksymmetric}
K(x) = K(-x), \qquad \text{ for all } x \in \Rd.
\end{equation} 
For symmetric potentials, the ODE system \eqref{eqn:so-discrete} and the kinetic equation \eqref{eq:kinetic} conserve 
the linear momentum: $\sum_{i=1}^N v_i$ and $\int_{\R^{2d}}v\Epsf(t,x,v) \dx \dv$, respectively. In terms of macroscopic variables, $\int_{\R^d} \Epsrho \Epsu \dx$ remains constant through the evolution of \eqref{eq:kinetic}, where $\Epsrho$ and $ \Epsu$ are defined by:
\[
\Epsrho(t,x) = \int_{\R^d}\Epsf(t,x,v) \dv,\quad
\Epsrho(t,x)\Epsu(t,x) = \int_{\R^d}v\Epsf(t,x,v) \dv.
\]
The key idea for dealing with the non-uniqueness of solutions to \eqref{eq:u-implicit} is to account for the fact that~\eqref{eq:macroscopic}  is the limit of a second-order model for which linear momentum conservation holds and hence, as a limiting equation, \eqref{eq:macroscopic} should inherit this momentum conservation as well. 

Given the considerations above, we append model \eqref{eq:macroscopic} with the constraint/requirement that the linear momentum remains constant through the time evolution. With no loss of generality, we assume that the linear momentum is zero at the initial time and hence, it has to remain zero for all times.  For ease of referencing, we list below the macroscopic model \eqref{eq:macroscopic} together with momentum conservation:
\begin{subequations}
\begin{gather}
    \del_t \rho \, + \Grad \cdot (\rho u) = 0, 
\qquad \rho \big|_{t=0} = \rho_0(x), \label{eq:rho-full}
\\
    \Phi(t,x) u(t, x) = \int_{\R^{d}} \phi(|x-y|) \rho(t,y) u(t,y) \dy 
                             - \Grad K \ast \rho(t,x),
\label{eq:u-implicit-full}
\\
   \int_{\R^d} \rho \, u \dx = 0. \label{eq:conserv-momentum}
\end{gather} 
\label{eq:macroscopic-full}
\end{subequations}
System \eqref{eq:macroscopic-full} is the main object of study of this paper.

The layout of the present work is as follows. In Section \ref{sect:well-posedness} we establish a well-posedness theory for system \eqref{eq:macroscopic-full}. Solutions are sought in the space $\CalP_c(\Rd)$ of probability measures  with compact support, using the mass transportation framework developed in \cite{CCR2011}. The main difficulty is to obtain a Lipschitz bound on $u$, given that for each $\rho$, \eqref{eq:u-implicit-full} provides multiple solutions. The momentum conservation condition \eqref{eq:conserv-momentum} plays an essential role, as it enforces uniqueness. The well-posedness result is stated in Theorem \ref{thm:macro}.

Another major issue of interest investigated in detail in this paper is the zero-inertia limit of solutions of the kinetic model \eqref{eq:kinetic}. We use again the framework from \cite{CCR2011} and consider measure-valued solutions $\Epsf$ of \eqref{eq:kinetic} in the mass transportation sense. For this reason, the method of characteristics plays a major role throughout our analysis. In Section \ref{sect:uniform-bounds} we establish uniform bounds in $\Eps$ of solutions to \eqref{eq:kinetic}. We note in particular the key bound on the second-order fluctuation term established by Theorem \ref{thm:G-2}. The uniform bounds enable us to derive the main convergence results in Section \ref{sect:limit}. Specifically, we show that $\Epsf$ converges weak-$^*$ as measures to $ f(t,x,v) = \rho(t,x) \, \delta(v - u(t, x))$ as $\Eps \to 0$, where $\rho$ (and $u$) represent the unique solution to the macroscopic model \eqref{eq:macroscopic-full}. The precise statements are presented in Theorems \ref{thm:limit} and \ref{thm:rho-trans}. We also point out the convergence of characteristics paths established by Theorem  \ref{thm:conv-traj}, where the subtlety resides in the singular limit, with characteristics of a second-order system collapsing into first-order characteristics.

Finally, in Section \ref{sect:numerics} we design and implement a numerical procedure for solving \eqref{eq:macroscopic}. The numerical method is rooted in the analytical considerations made for the continuum model, as a discrete analogue of the momentum conservation is enforced in order to obtain unique solutions for the discretization of \eqref{eq:u-implicit-full}. To test our method, we implement and illustrate it in one and two dimensions, for several choices of interaction kernels $K$ for which we know the exact steady states.

We note that for models that include only attraction and repulsion, the analogous zero inertia limits $\Eps \to 0$ of second-order models\footnote{We refer to the kinetic model \eqref{eq:kinetic} as being ``second-order" since it is based on Newton's second law \eqref{eqn:so-discrete}; in strict terms, equation \eqref{eq:kinetic} is a first-order PDE.} such as \eqref{eqn:so-discrete} and \eqref{eq:kinetic}, have been investigated recently in \cite{FS2014} (an earlier study, restricted to PDE analysis, was done in \cite{Jabin2000}). As in \cite{FS2014}, the focus in the present work is to find uniform (in $\Eps$) bounds on the size of the support of the solutions $\Epsf$, as well as to show the vanishing (with $\Eps$) of a second-order fluctuation term. By adding the alignment term such estimates become significantly more involved, as they require the use of the momentum conservation in \eqref{eq:kinetic}, as well as the coupling with control on high order moments of $\Epsf$.

Finally, we point out that  we work in this paper with smooth interaction potentials and influence functions. Specifically, we assume $\nabla K \in W^{1,\infty}(\Rd)$ and $\phi \in W^{1,\infty}(\R^+)$, where $\phi$ is also assumed to be positive, non-increasing, and to decay sufficiently slow at infinity. All these assumptions are used in an essential way in our analysis, in particular to control (uniformly in $\Eps$)  the support of the solutions to \eqref{eq:kinetic} and the second-order fluctuation term. Removing (some of) these requirements would require a significantly different approach and analysis than the present one.  We comment however that, from the point of view of applications, the smoothness requirement is mostly irrelevant, as in numerical simulations one does not normally distinguish between a singular potential and its smooth regularization. 

\section{Well-posedness of the macroscopic equation}
\label{sect:well-posedness}

In this section, we establish a well-posedness theory for the macroscopic equation ~\eqref{eq:macroscopic-full}.
Solutions $\rho$ are sought  in the space $C([0,T];\Pc(\Rd))$, where $\Pc(\Rd)$ denotes the space of probability measures with compact support in $\Rd$. We use the measure-transportation framework from \cite{CCR2011} and consider the characteristic equations associated to \eqref{eq:rho-full}, given by 
\begin{align} \label{eq:char-limit}
    \frac{\rm d}{\dt} x = u[\rho](t, x(t)), 
\qquad
    x(0) = x_0.
\end{align}
In order to specify the solution space for~\eqref{eq:macroscopic},  we use the concept of flow maps: suppose $\rho$ is a fixed function and the vector field $u [\rho](t, x)$ is locally Lipschitz in $x$. Then standard ODE theory provides the local existence of the flow map $\mathcal{T}^{t}_{u[\rho]}$ associated to \eqref{eq:char-limit}:
 \[
 x_0 \xrightarrow{\mathcal{T}^{t}_{u[\rho]}} x(t),
 \]
 where $x(t)$ is the unique solution of \eqref{eq:char-limit} that starts at $x_0$. 

\begin{Definition}
\label{defn:rho-transp}
Let $T>0$ be fixed. A measure-valued function $\rho \in C([0, T]; \Pc (\Rd))$ is said to be a solution of the macroscopic equation \eqref{eq:macroscopic-full} with initial condition $\rho_0 \in \Pc (\Rd)$ if
\begin{equation}
\label{eqn:rho-transp}
\rho(t) = \mathcal{T}^{t}_{u[\rho]} \# \rho_0,
\end{equation}
where at each $t \in [0,T]$, the velocity field $u[\rho](t, \cdot) \in W^{1, \infty}(\Supp\rho)$ satisfies \eqref{eq:u-implicit-full} and \eqref{eq:conserv-momentum}.
\end{Definition}
The push-forward in Definition \ref{defn:rho-transp} is taken in the mass transportation sense; equation \eqref{eqn:rho-transp} is equivalent to:
\[
\int_{\Rd} \varphi(x) \rho(t,x) dx = \int_{\Rd} \varphi ( \mathcal{T}^{t}_{u[\rho]}(X)) \rho_0(X) dX,
\]
for all $\varphi \in C_b(\Rd)$. 

The results in \cite{CCR2011} establish the existence and uniqueness of measure solutions for various aggregation models via fixed point arguments. These results apply in particular to attractive-repulsive models in the form \eqref{eq:rho-full}, but there the velocity field $u$ is given {\em explicitly} in terms of $\rho$. In \cite{CCR2011}, the assumptions made on the attractive-repulsive potential $K$ guarantee the local Lipschitz continuity of the velocity field $u$ and hence, the global existence and uniqueness of characteristics paths.

In our model, $u$  is given {\em implicitly} by \eqref{eq:u-implicit-full} and \eqref{eq:conserv-momentum}. While \eqref{eq:u-implicit-full} can have multiple solutions (as discussed above),  we show that the momentum conservation property \eqref{eq:conserv-momentum} enforces uniqueness. Once a unique solution $u[\rho]$ of \eqref{eq:u-implicit-full}-\eqref{eq:conserv-momentum} has been identified, this solution is shown to be Lipschitz continuous in the spatial variable so that the characteristic equation is well-posed. 
\medskip

We now establish the necessary a priori bounds for $(\rho, u)$ for deriving the well-posedness of solutions to \eqref{eq:macroscopic-full}. First we show the Lipschitz bound of $u$ for a fixed density function $\rho$. 
Similar estimates have also been used in \cite{MT2011, TT2014} for flocking models.
\begin{Proposition} \label{Prop:2-1} Assume $K$ is symmetric (i.e., it satisfies \eqref{eqn:Ksymmetric}), $ \Grad K \in W^{1, \infty}(\Rd)$ and $\phi \in W^{1,\infty}(\mathbb{R}^+)$ is positive and non-increasing. Consider a given probability measure $\rho(t,\cdot) \in \CalP_c(\R^d)$ at some fixed time $t$. 
Then equations~\eqref{eq:u-implicit-full}-\eqref{eq:conserv-momentum} uniquely define a bounded, Lipschitz continuous velocity field $u$ on the support of $\rho$. 
\end{Proposition}
\begin{proof}
Since equations (1.5b)-(1.5c) hold at a specific fixed time, we simplify the notation, and drop the time dependence in the proof. Denote
\begin{equation}
\label{eqn:supprho}
\Supprho = \max_{x \in  \Supp \rho} |x|.
\end{equation}
By our assumption on $\rho$, $\Supprho$ is finite. 
The following lower bound on the influence function $\phi$ can be inferred immediately:
\[\phi(|x-y|)\geq \eta \qquad\text{for
  any } x, y \in \Supp\rho,\]
where
\begin{equation}
\label{eqn:etadef}
\eta:=\phi(2\Supprho)>0.
\end{equation}
As a consequence, $\Phi(x)$ is bounded from below. Indeed, for any $x\in \Supp \rho$,
\begin{equation}\label{eq:Philow}
 \Phi(x) = \int_{\Supp\rho} \phi(|x-y|) \rho(y) \dy
                \geq \eta.
\end{equation}
Moreover, we have
\[\Phi(x) \leq \norm{\phi}_{L^\infty(0, \infty)}, \qquad \text{for any $x \in \R^d$.}\]

Define the following alignment operator $\CalA$:
\begin{equation}\label{eq:calA}
    \CalA[u](x) = \Phi(x) u(x) - \int_{\R^{d}} \phi(|x-y|) \rho(y) u(y) \dy.
\end{equation}
It is easy to check that $\CalA$ is a linear bounded operator which maps
$L^\infty(\Supp\rho)$ to $L^\infty(\Supp\rho)$.

Equation \eqref{eq:u-implicit-full} can now be expressed as
\[\CalA[u](x)=-\Gradx K\ast\rho(x).\]
As commented above, the operator $\CalA$ is {\em not} invertible since $\CalA$ maps all
constant (in $x$) functions to zero.

To overcome the degeneracy of $\CalA$, we make use of the momentum conservation condition
\eqref{eq:conserv-momentum} and define a new operator $\CalM$:
\begin{equation}\label{eq:calM}
\begin{aligned}
    \CalM[u](x) :&= \CalA[u](x)+\eta\int_{\R^d}\rho(y)u(y)dy \\
    &=\Phi(x)  u(x) - \int_{\R^{d}} (\phi(|x-y|)-\eta) \rho(y) u(y) \dy.
\end{aligned}
\end{equation}
We note that the operators $\CalA$ and $\CalM$ depend on $\rho$; however, we choose not to indicate this dependence explicitly in the notations, as this fact is irrelevant for the main well-posedness result that follows.

A solution $u$ of \eqref{eq:u-implicit-full} and \eqref{eq:conserv-momentum} also satisfies
\begin{equation}
\label{eqn:calM}
\CalM[u](x)=-\Gradx K\ast\rho(x).
\end{equation}
We claim that $\CalM$ is invertible. 
Indeed, given any $x\in\Supp\rho$,
\begin{align*}
\int_{\R^{d}} (\phi(|x-y|)-\eta) \rho(y) u(y) \dy~\leq&~
\|u\|_{L^\infty(\Supp\rho)}\int_{\Supp\rho}(\phi(|x-y|)-\eta)\rho(y)dy\\
=&~\|u\|_{L^\infty(\Supp\rho)}(\Phi(x)-\eta).
\end{align*}
By the definition of $\CalM$ in~\eqref{eq:calM}, one has
\[\CalM[u](x)\geq\Phi(x)u(x)-\|u\|_{L^\infty(\Supp\rho)}(\Phi(x)-\eta)
\qquad \quad \text{for all $x\in \Supp\rho$}.\]
For any $\delta > 0$, there exists $x_{\delta} \in \Supp\rho$ such that $u(x_\delta) \geq \| u\|_{L^\infty(\Supp\rho)} - \delta$ and hence,
\begin{align*}
    \CalM[u] (x_\delta) 
&\geq  
   \Phi(x_\delta) \|u\|_{L^\infty(\Supp\rho)} 
   - \|u\|_{L^\infty(\Supp\rho)}(\Phi(x_\delta)-\eta)
   - \delta \norm{\phi}_{L^\infty(0, \infty)}
\\
& = \eta \norm{u}_{L^\infty(\Supp\rho)} - \delta \norm{\phi}_{L^\infty(0, \infty)}.
\end{align*}
By letting $\delta \to 0$ we obtain
\begin{equation}\label{eq:CalMbound}
\|\CalM[u]\|_{L^\infty(\Supp\rho)}\geq\eta\|u\|_{L^\infty(\Supp\rho)}.
\end{equation} 
Hence for a given $\rho$, $\CalM$ is invertible on $L^\infty(\Supp\rho)$ and $u$ is uniquely solvable from \eqref{eqn:calM} such that
\begin{equation}\label{eq:ueq}
u(x)=(u[\rho])(x)=-\CalM^{-1}\left(\Gradx K\ast\rho(x)\right).
\end{equation}

Next we show that $u \in W^{1, \infty}(\Supp \rho)$ and find an explicit bound on its  $W^{1, \infty}$-norm. We start with the $L^\infty$ estimate. From \eqref{eq:CalMbound} and
\eqref{eq:ueq}, we have
\begin{align} \label{bound:u-1}
   \norm{u}_{L^\infty(\Supp\rho)} 
   \leq\|\CalM^{-1}\|~\|\Gradx
  K\ast\rho\|_{L^\infty(\Supp\rho)}\leq\frac{1}{\eta}\|\Gradx K\|_{L^\infty}.
\end{align}
Next, we estimate the bound of $\Gradx u$. The distributional derivative of $u$ satisfies
\begin{align*}
\Gradx u=\int_{\R^d}\Gradx \left(\frac{\phi(|x-y|)}{\Phi(x)}\right)\rho(y)u(y)dy
 +\Gradx\left(\frac{\Gradx K}{\Phi}\right).
\end{align*}
Applying the lower bound estimate \eqref{eq:Philow} for $\Phi$ and the
$L^\infty$ estimate \eqref{bound:u-1} for $u$
in each term gives
\begin{align*}
\left\|\int_{\R^d}\Gradx
 \left(\frac{\phi(|x-y|)}{\Phi(x)}\right)\rho(y)u(y)dy\right\|_{L^\infty(\Supp\rho)}
~\leq~ \frac{2}{\eta^2}\|\phi'\|_{L^\infty}\|\Gradx
  K\|_{L^\infty},
\end{align*}
\begin{align*}
\left\|\Gradx\left(\frac{\Gradx K}{\Phi}\right)\right\|_{L^\infty(\Supp\rho)}
\leq~\frac{1}{\eta^2}\|\phi'\|_{L^\infty}\|\Gradx
K\|_{L^\infty}+\frac{1}{\eta}\|\Gradx^2K\|_{L^\infty}.
\end{align*}
This shows 
\begin{align}\label{bound:Grad-u}
\|\Gradx u\|_{L^\infty(\Supp\rho)}\leq
\frac{3}{\eta^2}\|\phi'\|_{L^\infty}\|\Gradx
K\|_{L^\infty}+\frac{1}{\eta}\|\Gradx^2K\|_{L^\infty}.
\end{align}
Combining~\eqref{bound:u-1} with~\eqref{bound:Grad-u}, we have
\begin{align} \label{bound:Lip-u}
    \norm{u}_{W^{1, \infty}(\Supp\rho)}
\leq
    \vpran{\frac{1}{\eta} + \frac{3}{\eta^2} \norm{\phi'}_{L^\infty(0, \infty)}}
       \norm{\Grad K}_{L^\infty}
       + \frac{1}{\eta} \norm{\Grad^2 K}_{L^\infty} \,,
\end{align}
with $\eta$ defined in \eqref{eqn:etadef}.
\end{proof}

\begin{rmk}
\label{remark:outside} 
The values of the velocity field $u$ outside the support of $\rho$ do not explicitly enter in the dynamics of 
model \eqref{eq:macroscopic-full}. However,  we can extend the definition of $u$ outside of $\Supp\rho$ by using \eqref{eq:u-implicit-full}. This extension is applied in the proof  of Theorem \ref{thm:macro} where a $W^{1, \infty}$-bound on $u$ is used in a domain that is larger than $\Supp \rho$. In particular, we note that the right-hand-side of \eqref{eq:u-implicit-full} only depends on values of $\rho$ and $u$ {\em inside} the support of $\rho$. For any $R>0$ such that $\Supp \rho \subset B_R$ and $x \in B_R \setminus \Supp \rho$, $\Phi(x)$ is bounded from below by $\phi(2R)$ (see \eqref{eq:Philow}) and the assumptions made on $K$ and $\phi$ along with the bound \eqref{bound:Lip-u} immediately imply a bound on  $\norm{u}_{W^{1, \infty}(B_R)}$ that depends on $R$, $\norm{\Grad K}_{W^{1,\infty}}$ and $\norm{\phi}_{W^{1,\infty}}$ This bound has a similar expression as in \eqref{bound:Lip-u}, but with $\eta$ replaced by $\phi(2R)$. 
\end{rmk}

As discussed below, Proposition~\ref{Prop:2-1} combined with the fixed-point argument developed in \cite{CCR2011} is enough to conclude the {\em local} existence of solutions to \eqref{eq:macroscopic-full} in the measure-transportation sense of Definition \ref{defn:rho-transp}. In order to show the {\em global} existence, we need to prove {\em a priori} that the support of $\rho$ does not become infinite in a finite time. This requires an additional assumption on the influence function $\phi$. The precise statement is as follows:

\begin{Proposition} 
\label{Prop:2-2}
Assume $K$ and $\rho$ satisfy the hypotheses of Proposition \ref{Prop:2-1}, together with the following slow decay condition on the influence function $\phi$:
\begin{align} \label{cond:phi}
    \int_0^\infty \phi(r) \dr = \infty.
\end{align}
Let $(\rho, u)$ be a solution to ~\eqref{eq:macroscopic-full} with compactly supported initial density $\rho_0$.  Then for any $t>0$, we have
\begin{itemize}
\item [(a)] $\rho(t, \cdot)$ remains compactly supported and the bound of its support satisfies 
\begin{equation}
\label{eqn:supprho-est}
   \Supprho(t) \leq \frac{1}{2} \Psi^{-1} \vpran{\Psi (2 \Supprho(0)) + 2t \norm{\Grad K}_{L^\infty}},
\end{equation}
where $\Supprho(t)$ is defined in~\eqref{eqn:supprho} 
and $\Psi(r) = \int_0^r \phi(s) \ds$.  \smallskip

\item [(b)] $u\in W^{1,\infty}(\Supp\rho)$ and \eqref{bound:Lip-u} holds
with 
\[\eta(t)=\phi(\Psi^{-1}(\Psi (2 \Supprho(0)) + 2t \norm{\Grad K}_{L^\infty})).\]
\end{itemize}
\end{Proposition}
\begin{proof}
First, from~\eqref{bound:u-1} we know that the $L^\infty$-bound of $u$ depends on $\Supprho$. Specifically, 
\begin{align} \label{bound:u-2}
   \norm{u(t, \cdot)}_{L^\infty(\Supp\rho(t))} 
\leq  
    \frac{1}{\phi(2\Supprho(t))} \norm{\Grad K}_{L^\infty}
\end{align}
for all $t$ in a local interval of existence $[0, T]$. On the other hand, for any characteristic path $x(t)$  originating at $x_0 \in\Supp\rho_0$ (see \eqref{eq:char-limit}), we have
\[
\frac{d}{dt}|x(t)|\leq \|u(t,\cdot)\|_{L^\infty(\Supp\rho(t))}.
\]
Hence $\Supprho$ is bounded in terms of $u$ as
\begin{align*}
    \Supprho (t)  \leq \Supprho(0) + \int_0^t \norm{u(\tau,
  \cdot)}_{L^\infty(\Supp\rho(\tau))} \dtau.
\end{align*}
We can now close the loop between $\Supprho$ and $u$. To this end, denote
\begin{align*}
     D(t) = \Supprho(0) + \int_0^t \norm{u(\tau, \cdot)}_{L^\infty(\Supp\rho(\tau))} \dtau.
\end{align*}
Then $\Supprho(t) \leq D(t)$. By \eqref{bound:u-2} and the monotonicity of $\phi$,
\begin{align*}
     D'(t) = \|u(t,\cdot)\|_{L^\infty(\Supp\rho(t))}\leq
  \frac{1}{\phi(2 D(t))} \norm{\Grad K}_{L^\infty},
\end{align*}
which implies
\begin{align*}
    \frac{d}{dt} \Psi (2D(t)) \leq 2 \norm{\Grad K}_{L^\infty}.
\end{align*}
By integrating the equation above, we find
\begin{equation}
\label{eqn:Psi-est}
    \Psi(2 D(t)) \leq \Psi(2 \Supprho(0)) + 2 t \norm{\Grad K}_{L^\infty}.
\end{equation}
Assumption~\eqref{cond:phi} implies that $\Psi(r) \to \infty$ as $r \to
\infty$. Moreover, $\Psi(0)=0$ and $\Psi'>0$. Consequently, $\Psi$ is a bijection from $\R^+$ to $\R^+$. We can thus infer \eqref{eqn:supprho-est} from \eqref{eqn:Psi-est} and the inequality $\Supprho(t) \leq D(t)$.
The Lipschitz bound on $u$ in part (b) follows immediately.
\end{proof}

Combining the a priori bounds in Proposition \ref{Prop:2-2} with the fixed-point argument in \cite{CCR2011}, we can now show the global wellposedness of solutions to \eqref{eq:macroscopic-full}. The theorem states

\begin{Theorem}[Well-posedness of the macroscopic model] \label{thm:macro}
Assume $K$ is symmetric, $ \Grad K \in W^{1, \infty}(\Rd)$ and $\phi \in W^{1,\infty}(\mathbb{R}^+)$ is positive, non-increasing, and satisfies the slow decay condition \eqref{cond:phi}. Suppose the initial measure $\rho_0 \in \Pc(\Rd)$. 
Then for any $T > 0$, there exists a unique solution $(\rho, u)$
to~\eqref{eq:macroscopic-full} such that $\rho \in C([0, T], \Pc(\Rd))$ and $u \in C([0, T]; W^{1, \infty}(\Supp \rho))$.
\end{Theorem}

\begin{proof}
The proof follows directly from the theory developed in \cite{CCR2011}; we only summarize the setup and the main steps. We refer for details to the proof of Theorem 3.10 in \cite{CCR2011}. 

\medskip
{\em Local well-posedness.} Fix $R>0$ such that the ball $B_R$ contains the initial support of $\rho_0$, and consider the set $\setF$ consisting of all functions in $C([0,T],\Pc(\Rd))$ such that the support of $\rho_t(\cdot)=\rho(t,\cdot)$ is contained in the ball $B_R$ for all $t \in [0,T]$ ($T$ is a positive number to be chosen later). The set $\setF$ endowed with the distance $\mathcal{W}_1$ defined by
\[
\mathcal{W}_1(\rho_t,\tilde{\rho}_t) = \sup_{t \in [0,T]} W_1(\rho_t,\tilde{\rho}_t),
\]
is a complete metric space. Here, $W_1$ denotes the $1$-Wasserstein distance.

For fixed $\rho \in \setF$, one can define a map $\Fpmap$ on this metric space by
\[
\Fpmap [\rho](t) := \mathcal{T}^t_{u[\rho]} \# \rho_0.
\]
Consequently, solutions of \eqref{eq:macroscopic-full} in the sense of Definition \ref{defn:rho-transp} are fixed points of the map $\Fpmap$. The essential step is to show that this map is contractive, and hence, conclude that it has a unique fixed point in $\setF$. 

In our context, the key observation is that for every $\rho \in \setF$, one can use the argument in the proof of Proposition \ref{Prop:2-1} (with $\eta= \phi(2R)$ independent of $t\in [0,T]$ and $\rho \in \setF$) and find a $W^{1,\infty}$-bound on the velocity field $u[\rho]$, bound which depends only on $K$, $\phi$ and $R$ (see estimate \eqref{bound:Lip-u} and Remark \ref{remark:outside}).

The Lipschitz continuity of $u[\rho]$ provides well-posedness of characteristic trajectories for each $\rho \in \setF$. Also, for $T$ short enough, the support of $\Fpmap[\rho](t)$ is contained in $B_R$. By simple arguments (Lemmas 3.11 and 3.7 in \cite{CCR2011}), the following estimate can be shown to hold:
\begin{align*}
W_1(\Fpmap[\rho_1](t),\Fpmap[\rho_2](t)) \leq C(t) \sup_{t \in [0,T]}\| u[\rho_1] - u[\rho_2] \|_{L^\infty(B_R)}.
\end{align*}
To infer the contractivity of $\Fpmap$ we also need an estimate of the form
\begin{equation} 
   \sup_{t \in [0, T]} \norm{u[\rho_1] - u[\rho_2]}_{L^\infty(B_R)} \leq
   \tilde{C}(T) \sup_{t \in [0, T]} W_1(\rho_1, \rho_2),  \label{bound:u-local-1}
\end{equation}
for some constant $\tilde{C}$ that depends on $T$, as well as on $K$, $\phi$ and $R$. Requiring that $T$ is small enough such that $\tilde{C}<1$ concludes the existence of a fixed point of $\Fpmap$, and consequently, of a {\em local} solution to \eqref{eq:macroscopic-full} in the sense of Definition \ref{defn:rho-transp}.

For proving~\eqref{bound:u-local-1}, let $\rho_1, \rho_2 \in \setF$, and $\eta = \phi(2R)$. Denote $u_1 = u[\rho_1]$ and $u_2 = u[\rho_2]$. Then, by \eqref{eq:u-implicit-full} and \eqref{eq:conserv-momentum},
\begin{align*}
    (\phi \ast \rho_1) u_1(t, x) 
    = \int_{B_R} (\phi(|x-y|) - \eta) \rho_1(t,y) u_1(t,y) \dy 
                             - \Grad K \ast \rho_1(t,x),
\\
    (\phi \ast \rho_2) u_2(t, x) 
    = \int_{B_R} (\phi(|x-y|) - \eta) \rho_2(t,y) u_2(t,y) \dy 
                             - \Grad K \ast \rho_2(t,x).
\end{align*}
By Remark \ref{remark:outside}, we can extend $u_1$ and $u_2$ to the whole ball $B_R$ such that $u_1,u_2 \in W^{1,\infty}(B_R)$. Taking the difference of the above two equations, we have
\begin{equation}
\label{eqn:udiff}
   (\phi \ast \rho_1) (u_1 - u_2)
   = \int_{B_R} (\phi(|x-y|) - \eta) \rho_1(y) (u_1(y) - u_2(y)) \dy 
      + \CalR_1,
\end{equation}
where
\begin{align*}
\CalR_1
  = - \Grad K \ast (\rho_1 - \rho_2)
   - (\phi \ast (\rho_1 - \rho_2)) u_2
   + \int_{B_R} (\phi(|x-y|) - \eta) (\rho_1(y) - \rho_2(y)) u_2(y) \dy \,.
\end{align*}
By Lemma 3.15 in \cite{CCR2011}, we have
\begin{align*}
    \norm{\CalR_1}_{L^\infty(B_R)}
\leq
   \left(\norm{\Grad^2 K}_{L^\infty(\R^d)}
          + \norm{\phi}_{W^{1,\infty}(0,\infty)} 
             \norm{u_2}_{W^{1,\infty}(B_R)}\right)
    W_1(\rho_1, \rho_2) \,,
\end{align*}
where a bound on $ \norm{u_2}_{W^{1,\infty}(B_R)}$  can be inferred as explained in Remark \ref{remark:outside}. 

From \eqref{eqn:udiff}, following an argument similar to that leading to \eqref{bound:u-1} from \eqref{eq:calM} and \eqref{eqn:calM}, we estimate
\begin{align*}
   \norm{u_1(t,\cdot) - u_2(t,\cdot)}_{L^\infty(B_R)} 
\leq 
   \tilde{C} W_1(\rho_1, \rho_2), \qquad \quad \text{ for all } t \in [0,T],
\end{align*}
where $\tilde{C}$ depends on $T$, $R$, $\norm{\phi}_{W^{1,\infty}(0,\infty)}$ and $\norm{\Grad K}_{W^{1,\infty}(\R^d)}$. Hence~\eqref{bound:u-local-1} holds and there exists $T > 0$ small enough such that~\eqref{eq:macroscopic} has a unique solution $(\rho, u)$ over the time interval $[0, T]$. 

\medskip
{\em Global well-posedness.} Note that the argument above does not use the slow decay condition \eqref{cond:phi} on $\phi$. Local existence can be extended for as long as the support of $\rho$ remains bounded. Provided that $\phi$ satisfies the additional assumption \eqref{cond:phi}, Proposition \ref{Prop:2-2} guarantees that the support of $\rho$ does not blow up in a finite time. In this case the local theory can be extended to $[0, T]$ for any $T > 0$, which gives the global well-posedness.  
\end{proof}

\begin{rmk}
\label{remark:slow-decay}
It is interesting to note that the slow decay condition \eqref{cond:phi} appears in previous literature on the Cucker-Smale model, in connection to asymptotic flocking. Indeed, for influence functions in power-law form, it has been shown in \cite{HL2009} that the discrete Cucker-Smale model has the \emph{unconditional flocking} property when $\phi$ satisfies \eqref{cond:phi}. In addition, the same property \eqref{cond:phi} is needed for unconditional flocking in the corresponding kinetic and macroscopic Cucker-Smale models \cite{CFRT2010,Tan2014,TT2014}. In our context, condition \eqref{cond:phi} is first needed to guarantee global-in-time well-posedness of solutions to \eqref{eq:macroscopic-full}, as discussed above. Then, condition \eqref{cond:phi} is needed again, in an essential way, to show that solutions $\Epsf$ of the kinetic model \eqref{eq:kinetic} remain compactly supported for all finite times, with a bound on their support that is independent of $\Eps$ (see Proposition \ref{prop:unif-supp} and its proof). This estimate is essential in the $\Eps \to 0$ limit of the kinetic model.
\end{rmk}

Note that by \cite{AGS2006} (Lemma 8.1.6 in Chapter 8),  the unique measure-valued solution in Theorem~\ref{thm:macro} is exactly the unique weak solution for~\eqref{eq:macroscopic-full}, i.e., it satisfies

\begin{align} \label{eq:rho-weak}
   \int_0^T \int_{\R^d} \del_t \varphi(t, x) \rho(t, x) \dx\dt
   + \int_0^T \int_{\R^d} \Grad \varphi \cdot \rho u \dx\dt
   + \int_{\R^d} \varphi(0, x) \rho_0(x, v) \dx=0,
\end{align}
for any $\varphi \in C^1_c([0, T); C^1_b(\R^d))$. 


\section{Well-posedness and uniform bounds of the kinetic model}
\label{sect:uniform-bounds}
In this section we make preparations for proving the convergence of the kinetic equation \eqref{eq:kinetic} to the macroscopic model \eqref{eq:macroscopic-full} by showing several uniform bounds in $\Eps$. These include an uniform bound on the support of the density function $\Epsf$, along with uniform bounds on its moments. A key step is to show that 
a certain ``second-order fluctuation" term vanishes with $\Eps$, which is the main ingredient that allows the passage from the second-order kinetic equation \eqref{eq:kinetic} to the first order macroscopic model \eqref{eq:macroscopic-full}.


 \subsection{Well-posedness setup for the kinetic model \eqref{eq:kinetic}}
 Model \eqref{eq:kinetic} with $\Eps>0$ fixed is a particular case of the general class of models studied in  \cite{CCR2011}. For completeness, we present briefly the well-posedness setup from \cite{CCR2011}, which is similar in fact to the measure transportation framework used for the macroscopic equation in Section \ref{sect:well-posedness}.
 
The solution space for equation \eqref{eq:kinetic} is $C([0,T],\Pc(\Rd \times \Rd))$, where $\Pc(\mathbb{R}^d \times \Rd)$ represents the space of probability measures with compact support in $\Rd \times \Rd$. The approach is to consider the characteristic equations associated with \eqref{eq:kinetic} and define measure-valued solutions $ \Epsf$ to \eqref{eq:kinetic} as the mass-transport (or push-forward) of the initial density $f_0$ along the flow defined by characteristics. 

 Specifically, the characteristic equations for \eqref{eq:kinetic} are given by:
 \begin{align} \label{eq:characteristics}
     &\frac{{\rm d} x}{\dt} = v, \nn
 \\
    &\frac{{\rm d} v}{\dt} 
      = - \frac{1}{\Eps} 
        \vpran{\int_{\R^d}\int_{\R^d} 
                        \phi(|x-y|) (v - v^\ast) \Epsf(t, y, v^\ast) \dy\dv^\ast
                    + \Grad K \ast \rho_\Eps(t, x)},
 \\
   & (x, v) \big|_{t=0} = (x_0, v_0) \in {\rm supp}f_0. \nn
 \end{align}
 
 In short-hand notation, one can write the characteristic system \eqref{eq:characteristics} as
 \begin{align} 
  \label{eq:characteristics-gen}
     &\frac{{\rm d} x}{\dt} = v, \nn
 \\
    &\frac{{\rm d} v}{\dt} 
      = \mathcal{H} [\Epsf](t, x, v),
 \\
   & (x, v) \big|_{t=0} = (x_0, v_0) \in {\rm supp}f_0, \nn
 \end{align}
 where $\mathcal{H}[\Epsf]$ denotes the entire right-hand-side of the $v$-equation in \eqref{eq:characteristics}. 

 Suppose that $\Epsf$ is a fixed function and the vector field $\mathcal{H} [\Epsf](t, x, v)$ is locally Lipschitz in $x$ and $v$. Then, standard ODE theory provides the local existence of the flow map $\mathcal{T}^{t,\Eps}_{\mathcal{H}[\Epsf]}$ associated to \eqref{eq:characteristics-gen}:
 \[
 (x_0,v_0) 
  \xrightarrow{\mathcal{T}^{t,\Eps}_{\mathcal{H}[\Epsf]}} (x,v), \quad \qquad (x,v) = (x_\Eps(t),v_\Eps(t)),
 \]
 where $(x_\Eps(t),v_\Eps(t))$ is the unique solution of \eqref{eq:characteristics-gen} that starts at $(x_0,v_0)$. 

Similar to Definition \ref{defn:rho-transp}, we consider measure-valued solutions of \eqref{eq:kinetic}  in the mass-transportation sense \cite{CCR2011}:
\begin{Definition}
\label{defn:sol}
Let $T>0$ be fixed. A measure-valued function $\Epsf\in C([0, T]; \Pc (\Rd \times \Rd))$ is said to be a solution of the kinetic equation \eqref{eq:kinetic} with initial condition $f_0 \in \Pc (\Rd \times \Rd)$ if
\begin{equation}
\label{eqn:f-transp}
\Epsf(t) = \mathcal{T}^{t,\Eps}_{\mathcal{H}[\Epsf]} \# f_0,
\end{equation}
for each $t \in [0,T]$.
\end{Definition}
The push-forward in \eqref{eqn:f-transp} is in the mass transportation sense:
\begin{equation}
\label{eqn:trans1}
\int_{\R^{2d}} \zeta (x,v) \Epsf (t,x,v) dx dv = \int_{\R^{2d}} \zeta(\mathcal{T}^{t,\Eps}_{\mathcal{H}[\Epsf]}(X,V)) f_0(X,V) dX dV,
\end{equation}
for all $\zeta \in C_b(\Rd \times \Rd)$.

The results in \cite[Section 4]{CCR2011}  apply to the kinetic equation \eqref{eq:kinetic} with $\Eps>0$ fixed.  The well-posedness result that can be inferred directly from there is the following.

 \begin{Theorem}[Well-posedness of the kinetic model \eqref{eq:kinetic} \cite{CCR2011}] 
 \label{thm:exist-measure}
 Assume the following properties on the influence function $\phi$ and the interaction potential $K$:
 \[
 \phi \text{ and }  \nabla K \text{ are locally Lipschitz, and } \; |\phi(x)|, |\nabla K(x)| \leq C(1+|x|) \text{ for all } x\in \Rd,
 \]
 for some $C>0$.  Take an initial measure $f_0 \in \Pc(\Rd \times \Rd)$. Then there exists a unique solution $\Epsf \in C([0,\infty), \Pc(\Rd \times \Rd))$ of \eqref{eq:kinetic}, in the sense of Definition \ref{defn:sol}, whose support grows at a controlled rate. Specifically, there exists an increasing function $R_\Eps(T)$ such that for all $T>0$,
 \[
 \operatorname{supp }\Epsf(t) \subset B_{R_\Eps(T)} \quad  \text{ for all } t \in [0,T].
 \]
 \end{Theorem}
 
We also note that by results in \cite{AGS2006,CaFrFiLaSl2011}, $f_\Eps$ given by Theorem \ref{thm:exist-measure} is also a weak solution of \eqref{eq:kinetic}, i.e., it satisfies
 \begin{equation} \label{eq:f-Eps-weak}
 \begin{aligned}
 & \int_0^T \int_{\R^{2d}} \del_t \psi(t, x, v) \Epsf(t, x, v) \dx\dv \dt
     + \int_0^T \int_{\R^{2d}} \Grad \psi \cdot v \Epsf \dx\dv\dt
 \\
 & \qquad - \frac{1}{\Eps} 
          \int_0^T \int_{\R^{2d}} 
              \nabla_v \psi \cdot F[\Epsf] \Epsf \dx\dv\dt
     + \int_{\R^{2d}} \psi(0, x, v) f_0(x, v) \dx\dv=0,
 \end{aligned}
 \end{equation}
 for any $\psi \in C^1_c([0, T); C^1_b(\R^d \times \R^d))$.  
 
\begin{rmk}
\label{rmk:zerofmom}
An essential assumption made throughout our work is that the interaction potential $K$ is symmetric. For such potentials, the kinetic model \eqref{eq:kinetic} conserves linear momentum through time evolution. To see this, one can take in the weak formulation  \eqref{eq:f-Eps-weak}  the test function $\psi(t, x, v) = \psi_1(t) v \psi_2(v)$, where $\psi_1 \in C^1_c(0,T)$ and 
$v \psi_2(v) \in C^1_b(\Rd)$, with $\psi_2=1$ on the set $\{ v \in \Rd: (x,v) \in \Supp \Epsf \} $. Due to the symmetry of $K$,  $\int_{\R^{2d}} F[\Epsf] \Epsf \dx\dv = 0$,  and we obtain that
\begin{align*}
    \frac{\rm d}{\dt} \int_{\R^{2d}} v \Epsf(t, x,v) \dx \dv = 0,
\qquad
   \text{in the distributional sense.}
\end{align*}

Throughout the paper we assume, with no loss of generality, that the initial density $f_0$ has zero momentum, i.e.,
\begin{equation}
\label{eqn:init-fm-zero}
\int_{\R^{2d}} v f_0(x,v) \dx \dv =0.
\end{equation}
Hence, the solutions $\Epsf$ provided by Theorem  \ref{thm:exist-measure} satisfy:
\begin{equation}
\label{eqn:feps-zeromom}
\int_{\R^{2d}} v \Epsf(t,x,v) \dx \dv =0,
\end{equation}
for all finite times.
\end{rmk}
 

\subsection{Uniform bound of the support} We prove here that starting from a compactly supported probability density function $f_0$, the solution $\Epsf(t, \cdot, \cdot)$ to \eqref{eq:kinetic} remains compactly supported in $(x, v)$ for any finite time. Moreover, the size of $\Supp \Epsf$ is independent of $\Eps$. 

We first prove a rough bound on the first moment\footnote{We adopt here a common terminology used for kinetic equations and refer to {\em moments} of $\Epsf$ the moments with respect with the velocity variable $v$.} of $\Epsf$.

\begin{Proposition} \label{prop:3-1} Let $\Epsf$ be the solution to~\eqref{eq:kinetic}, as provided by Theorem  \ref{thm:exist-measure}. Suppose $\phi \geq 0$ and that the initial data $f_0$ has a finite first moment. Then for each finite time $T>0$ and each fixed $\Eps>0$, $\Epsf(t)$ has a finite first moment for any $t \in [0, T]$. The bound of the first moment only depends on $T, \Eps$, and $K$. In particular, it is independent of $\phi$. 
\end{Proposition}
\begin{proof}
By symmetry of the integral we have
\begin{align*}
    \int_{\R^{2d}} \frac{v}{|v|} \cdot F[\Epsf] \Epsf \dx\dv
&    = \int_{\R^{4d}} 
           \phi(|x-y|) (v - v^\ast) \cdot \frac{v}{|v|}
            \Epsf(t, y, v^\ast) \Epsf(t, x, v) \dx\dy\dv\dv^\ast
\\
& \quad \,
       + \int_{\R^{2d}} \frac{v}{|v|} 
                  \cdot \vpran{\Grad K \ast \Epsf} \Epsf \dx\dv
\\
& \hspace{-1.5cm} 
     = \frac{1}{2} \int_{\R^{4d}} 
           \phi(|x-y|) (v - v^\ast) \cdot \vpran{\frac{v}{|v|} - \frac{v^\ast}{|v^\ast|}}
            \Epsf(t, y, v^\ast) \Epsf(t, x, v) \dx\dy\dv\dv^\ast
\\
& \quad \,
       + \int_{\R^{2d}} \frac{v}{|v|} 
                  \cdot \vpran{\Grad K \ast \Epsf} \Epsf \dx\dv.
\end{align*}
Notice that
\begin{align*}
   (v - v^\ast) \cdot \vpran{\frac{v}{|v|} - \frac{v^\ast}{|v^\ast|}}
= |v| + |v^\ast| - \frac{v \cdot v^\ast}{|v|} - \frac{v^\ast \cdot v}{|v^\ast|}
\geq 0.
\end{align*}
Hence,
\begin{align*}
    \int_{\R^{2d}} \frac{v}{|v|} \cdot F[\Epsf] \Epsf \dx\dv
\geq 
    \int_{\R^{2d}} \frac{v}{|v|} 
                  \cdot \vpran{\Grad K \ast \Epsf} \Epsf \dx\dv.
\end{align*}
This gives
\begin{align*}
   \frac{\rm d}{{\rm d}t} \int_{\R^{2d}} |v| \Epsf(t, x, v) \dx\dv
\leq 
  - \frac{1}{\Eps}\int_{\R^{2d}} \frac{v}{|v|} 
                  \cdot \vpran{\Grad K \ast \Epsf} \Epsf \dx\dv
\leq \frac{1}{\Eps}\norm{\Grad K}_{L^\infty}.
\end{align*}
Therefore, for $t \in [0, T]$, we have
\begin{align*}
     \int_{\R^{2d}} |v| \Epsf(t, x, v) \dx\dv
\leq 
     \int_{\R^{2d}} |v| f_0(x, v) \dx\dv
     + \frac{T}{\Eps} \norm{\Grad K}_{L^\infty}.
\end{align*}
Hence the first moment of $\Epsf$ is bounded for each $\Eps$ and $T$, and the bound is independent of $\phi$. 
\end{proof}

The following proposition shows that $\Epsf(t, \cdot, \cdot)$ is compactly supported and the size of its support is independent of $\Eps$.
\begin{Proposition}[Uniform bound on the support of $\Epsf$]
\label{prop:unif-supp}
 Suppose $K$ is symmetric, $\Grad K \in
W^{1, \infty}(\Rd)$ and $\phi \in W^{1,\infty}(\mathbb{R}^+)$ is positive,
non-increasing, and satisfies the slow decay condition \eqref{cond:phi}.
 Let $\Epsf$ be the solution to~\eqref{eq:kinetic} with an initial data $f_0 \in \Pc(\Rd \times \Rd)$ that satisfies \eqref{eqn:init-fm-zero}. Then for each finite time $T$, $\Epsf(t, \cdot, \cdot)$ has a compact support for any $t \in [0, T]$. The size of the support only depends on $T, \phi$, and $K$; in particular, it is independent of $\Eps$. 
\end{Proposition}
\begin{proof}
Consider the characteristic equations \eqref{eq:characteristics} associated with~\eqref{eq:kinetic}. Using the boundedness of the first moment of $\Epsf$ on the time interval $[0, T]$, as shown in Proposition~\ref{prop:3-1}, we immediately derive from~\eqref{eq:characteristics} that for each fixed $\Eps$ and $T$, $\Epsf$ has a compact support in $(x, v)$. Our goal is to show that the size of this compact support is independent of $\Eps$. 

To indicate the dependence on the initial location, we use the notation $(\Epsx(t; x_0, v_0),\Epsv(t; x_0, v_0))$ for the characteristic path that originates at $(x_0,v_0)$. 
Consider now the entire family of characteristic paths initialized at points $(x_0,v_0)$ inside the support of $f_0$. Since $\Epsx(t; x_0, v_0)$ and $\Epsv(t; x_0, v_0)$ are continuous in $(x_0, v_0)$ for each $t \in [0, T]$,  the maxima of the supports of $\Epsx$ and $\Epsv$ are attained, and one can define the sizes $\Suppxeps(t)$ and $\Suppveps(t)$ of these supports:
\begin{align} \label{def:Supp-x-v}
   \Suppxeps(t) = \max_{(x_0,v_0)\in {\rm supp}f_0 }|x_\Eps(t; x_0,v_0)|, 
\qquad
  \Suppveps(t) = \max_{(x_0,v_0) \in {\rm supp}f_0}|v_\Eps(t; x_0,v_0)|.  
\end{align}
Denote
\begin{align} \label{def:U}
    \EpsU(t) = \Suppxeps(0) + \int_0^t \abs{\Epsv(\tau)} \dtau.
\end{align}
Then 
\begin{align} \label{bound:1}
    \Suppxeps(t) \leq \EpsU(t).
\end{align}
Denote
\begin{equation}
\label{eqn:PhiEps}
    \Phi_\Eps(t,x) = \int_{\R^d} \phi(|x-y|) \rho_\Eps(t,y) \dy.
\end{equation}
Let $\Vmax(t)$ be the velocity with the largest magnitude at time $t$, that is,
\begin{equation}
\label{eqn:vM}
     \abs{\Vmax(t)} = \Suppveps(t),
\end{equation}
and let $\eta(t) = \phi(2D_\Eps(t))$. Then the monotonicity of $\phi$ gives
\begin{align*}
    0 < \eta(t) \leq \phi(|x - y|)
\qquad 
    \text{ for any } (x,\cdot), (y,\cdot) \in {\Supp} \Epsf (t, \cdot, \cdot).
\end{align*}
The time evolution of $\Suppveps$ is given by
\begin{align} 
\label{ineq:abs-v}
   \frac{{\rm d}\Suppveps}{\dt} 
 &  = - \frac{1}{\Eps}  \frac{\Vmax}{|\Vmax|} \cdot
       \vpran{\int_{\R^{2d}} 
                       \phi(|\Epsx-y|) (\Vmax - v^\ast) \Epsf(t, y, v^\ast) \dy\dv^\ast
                   + \Grad K \ast \rho_\Eps} \nn
\\
   &  \leq 
       - \frac{1}{\Eps} 
         \vpran{|\Vmax| \int_{\R^{d}} \phi(|\Epsx-y|) \rho_\Eps(t, y) \dy
                    - \frac{\Vmax}{|\Vmax|} \cdot\int_{\R^{2d}} \phi(|\Epsx-y|) v^\ast \Epsf(t, y, v^\ast) \dy\dv^\ast}  
\\
  & \quad + \frac{1}{\Eps} \norm{\Grad K}_{L^\infty} \Denote RHS. \nn
\end{align}
By the momentum conservation \eqref{eqn:feps-zeromom}, we have
\begin{align*}
 RHS 
  & =
       - \frac{1}{\Eps} 
         \vpran{|\Vmax| \Phi_\Eps(\Epsx)
                  - \frac{\Vmax}{|\Vmax|} \cdot
                      \int_{\R^{2d}} (\phi(|\Epsx-y|) - \eta(t)) v^\ast \Epsf(t, y, v^\ast) \dy\dv^\ast} 
 + \frac{1}{\Eps} \norm{\Grad K}_{L^\infty} \nn
 \\
   & \leq 
       - \frac{1}{\Eps} 
         \vpran{|\Vmax| \Phi_\Eps(\Epsx)
                  - \abs{\Vmax}
                      \int_{\R^{2d}} (\phi(|\Epsx-y|) - \eta(t)) \Epsf(t, y, v^\ast) \dy\dv^\ast} 
  + \frac{1}{\Eps} \norm{\Grad K}_{L^\infty} \nn
\\
   & =                    
       - \frac{\eta(t)}{\Eps} 
         \vpran{|\Vmax| - \frac{\norm{\Grad K}_{L^\infty}}{\eta(t)}}.  \nn
\end{align*}
Using the estimate above in \eqref{ineq:abs-v}, along with \eqref{eqn:vM} and the fact that $\eta(t)$ is non-increasing, we find
\begin{align*}
   \frac{\rm d}{\dt} \vpran{\Suppveps - \frac{\norm{\Grad K}_{L^\infty}}{\eta(t)} }
&\leq
       - \frac{\eta(t)}{\Eps} 
         \vpran{|\Vmax| - \frac{\norm{\Grad K}_{L^\infty}}{\eta(t)}}
       - \frac{\rm d}{\dt} \vpran{\frac{\norm{\Grad K}_{L^\infty}}{\eta(t)}}.  \nn
\\
&\leq
       - \frac{\eta(t)}{\Eps} 
         \vpran{\Suppveps - \frac{\norm{\Grad K}_{L^\infty}}{\eta(t)}} \,.
\end{align*}
Therefore for $t \in [0, T]$ we have
\begin{align} \label{bound:Supp-v}
   \Suppveps(t) 
\leq 
   \Suppveps(0) - \frac{\norm{\Grad K}_{L^\infty}}{\eta(0)}
   + \frac{\norm{\Grad K}_{L^\infty}}{\eta(t)}
\leq 
   \Suppveps(0) 
   + \frac{\norm{\Grad K}_{L^\infty}}{\eta(t)} \,.
\end{align}
Combining~\eqref{bound:Supp-v} with the definition of $\EpsU(t)$ in~\eqref{def:U}, we obtain the differential inequality 
\begin{align} \label{ineq:U-Eps}
    \EpsU'(t) 
&\leq 
   \Suppveps(0) + \frac{\norm{\Grad K}_{L^\infty}}{\eta(t)}
   = \Suppveps(0) + \frac{\norm{\Grad K}_{L^\infty}}{\phi(2D_\Eps(t))}.
\end{align}
Note that $\EpsU$ is an increasing function in $t$ and $\phi(r) \downarrow 0$ as $r \to \infty$. To show that $\EpsU$ is bounded independent of $\Eps$,
we let 
\begin{align*}
    \Lambda(r) =  \int_0^r \frac{\phi(s)}{c_2 + c_1 \phi(s)} \ds,
\end{align*}
where $c_1 = \Suppveps(0)$ and $c_2 = \norm{\Grad K}_{L^\infty}$. Then $\Lambda$ is an increasing function and it has the bounds
\begin{align*}
   \frac{1}{c_2 + c_1 \norm{\phi}_{L^\infty}}\int_0^r \phi(s) \ds
\leq \Lambda(r) \leq 
   \frac{1}{c_2}\int_0^r \phi(s) \ds.
\end{align*}
By~\eqref{cond:phi} we have
${\rm Range} (\Lambda) = [0, \infty)$. Meanwhile, for each fixed $\Eps$ and $T$ such that $\EpsU$ is well-defined on $[0, T]$, solving~\eqref{ineq:U-Eps} gives
\begin{align*}
    \Lambda(2\EpsU(t)) 
\leq 
   \Lambda(2\EpsU(0)) 
    + 2 t.
\end{align*}
Note that $\EpsU(0)=\Suppxeps(0)$ is independent of $\Eps$, as it represents the size of the spatial projection of the support of $f_0$ (see \eqref{def:Supp-x-v}). Therefore we obtain the following uniform in $\Eps$ bound:
\begin{align*}
    \EpsU(t) 
\leq \frac{1}{2}\Lambda^{-1} \vpran{\Lambda(2|\Supp f_0|) + 2 T}
\qquad \quad
   \text{for any $t \in [0, T]$}.
\end{align*}
We conclude that by~\eqref{bound:1}, we have that $ \Suppxeps(t)$ is uniformly bounded in $\Eps$ on $[0, T]$.  Moreover, by~\eqref{bound:Supp-v} and the definition of $\eta$, we also have that the support of $v$ is uniformly bounded in $\Eps$ on $[0, T]$. 
\end{proof}

Throughout the rest of the paper we denote by $\Omega(T) \subset \R^d \times \Rd$  the common support of $\Epsf(t,\cdot,\cdot)$ for all $\Eps>0$ and $t \in [0, T]$. Also, we use the following notation for the spatial and velocity projections of $\Omega(T)$:
\begin{align} \label{def:supp-x-v}
\Omega_1(T) = \{ x \in \R^d: (x,v) \in \Omega(T) \} \,,
\qquad
\Omega_2(T) = \{ v \in \R^d: (x,v) \in \Omega(T) \} \,.
\end{align}
\smallskip

Using the uniform bound of $\Epsv$, we can immediately infer the uniform (in $\Eps$) bounds for all the moments of $\Epsf$: 
\begin{Corollary} \label{lem:moments}
Suppose the assumptions in Proposition \ref{prop:unif-supp} hold. Then for any $k \in \N$ there exists a constant $C(T, k, \phi, K)$ such that
\begin{align} \label{bound:moments-general}
    \int_{R^{2d}} |v|^k \Epsf(t, x, v) \dx\dv
\leq C(T, k, \phi, K)
< \infty \,, \qquad k \geq 1 \,.
\end{align}
\end{Corollary}

\begin{rmk}
By doing a more careful estimate,  one can show that~\eqref{bound:moments-general} holds as long as
the initial measure $f_0$ has bounded moments (with no requirement on $f_0$ having compact support). 
Nevertheless,  we content ourselves with Corollary~\ref{lem:moments} 
since the main setting here requires that the initial data is compact supported.
\end{rmk}


\subsection{Uniform bound of the fluctuation} In this part we show a uniform in $\Eps$ bound on the second order fluctuation term defined below by \eqref{def:G-2}.  This bound (and in fact the rate of vanishing of the second order fluctuation with $\Eps$) is essential to show the convergence of $\Epsf$ \cite{Jabin2000, FS2014}. The estimates here are a priori. However, one can apply a density argument (see for example, Section 4 in \cite{FS2014}) to show that they hold for measure-valued solutions $\Epsf(t, \cdot, \cdot) \in \CalP_c(\Rd \times \Rd)$ at each $t \in [0, T]$. 

By previous estimates, we have a priori uniform-in-$\Eps$ bounds for all the moments and the size of the support of $\Epsf$ over an arbitrary finite time interval $[0, T]$. Hence, for $T>0$ fixed,  there exists a constant $C_0(T)$ that depends on $\phi$ and $K$ (but independent of $\Eps$) such that, for $k=1,2$,
\begin{align} \label{bound:moment-f-k-supp}
     \int_{\R^{2d}} |v|^k \Epsf(t, x, v) \dx\dv + \abs{\Supp \Epsf (t,\cdot, \cdot)} \leq C_0(T), \qquad \text{ for all } t \in [0,T].
\end{align}
Also,  there exists a constant $\tilde C_0(T)$ independent of $\Eps$ such that 
\begin{align} \label{bound:phi-T}
     \phi(\abs{x-y}) \geq \tilde C_0(T) > 0,
\qquad \text{for any $x, y \in \Supp \Epsf$, and $ t \in [0, T]$.}
\end{align}

\begin{Theorem} \label{thm:G-2}
Suppose the interaction potential $K$ and the influence function $\phi$ satisfy the assumptions in Proposition \ref{prop:unif-supp}. Let $\Epsf$ be the solution to~\eqref{eq:kinetic} starting from an initial data $f_0 \in \Pc(\Rd \times \Rd)$ that satisfies \eqref{eqn:init-fm-zero}.  Let $\Gfluc$ be the second order fluctuation defined by 
\begin{align} \label{def:G-2}
     \Gfluc(t) = \int_{\R^{2d}} \abs{F[\Epsf]}^2 \Epsf \dx\dv.
\end{align}
Then for any $t \in [0, T]$, there exists $C(T)$ independent of $\Eps$ such that
\begin{align*}
     \Gfluc(t) 
\leq 
    C(T) \, \Eps^2 + \Gfluc(0) \, e^{-\frac{2}{\Eps} t}.
\end{align*}
\end{Theorem}

\begin{proof}
We first make the simple observation that by Cauchy-Schwartz and \eqref{bound:moment-f-k-supp},  for $k=0,1$, we have
\begin{align} \label{bound:F-f}
    \int_{\R^{2d}} \abs{v}^k \abs{F[\Epsf]} \Epsf \dx\dv
\leq \vpran{\int_{\R^{2d}}|v|^{2k} \Epsf \dx\dv}^{1/2} \Gfluc(t)^{1/2}
< C_0(T) \, \Gfluc(t)^{1/2}.
\end{align}
Now we estimate the rate of change of $\Gfluc$. To this end, we multiply $\frac{1}{2}\abs{F[\Epsf]}^2$ to equation~\eqref{eq:f-Eps} and integrate in $x, v$.
This gives
\begin{align}
\label{eqn:dG2dt}
    \frac{1}{2}\frac{\rm d}{{\rm d}t} \Gfluc
& = \underbrace{\frac{1}{2} \int_{\R^{2d}} v\Epsf \cdot \Gradx \vpran{\abs{F[\Epsf]}^2}  \dx\dv}_{R_1}
    + \underbrace{\frac{1}{2} \int_{\R^{2d}} \vpran{\del_t \abs{F[\Epsf]}^2} \Epsf \dx\dv}_{R_2} \\
& \quad \underbrace{- \frac{1}{2\Eps} \int_{\R^{2d}} \Gradv\vpran{\abs{F[\Epsf]}^2} \cdot \vpran{F[\Epsf] \Epsf} \dx\dv}_{D_0}. \nn
\end{align}
Below we estimate separately the terms $R_1, R_2$ and $D_0$ indicated above. 

\smallskip
\noindent\underline{Estimate of $D_0$:} We have
\begin{align}
\label{eqn:D0}
    D_0 
    & = - \frac{1}{\Eps} \int_{\R^{2d}}
                \Gradv F : (F \otimes F) \Epsf \dx\dv \nn \\
    & = -\frac{1}{\Eps} \int_{\R^{2d}} \Phi_\Eps(x-y) \abs{F[\Epsf]}^2 \Epsf \dx\dv,
\end{align}
where for the second equality we used \eqref{eq:FfEps} and the definition \eqref{eqn:PhiEps} of $\Phi_\Eps(x)$. By \eqref{bound:phi-T} we conclude that $D_0$ provides dissipation in \eqref{eqn:dG2dt}; the exact expression of $D_0$ is used below to control a certain contribution from $R_2$.

\smallskip
\noindent\underline{Estimate of $R_1$:} From the expression of $R_1$ and \eqref{eq:FfEps}, we infer immediately:
\begin{align*}
    R_1 
    \leq \int_{\R^{2d}} \abs{\Gradx F[\Epsf]} \abs{F[\Epsf]} \abs{v} \Epsf \dx\dv,
\end{align*}
where 
\begin{align*}
  \abs{ \Gradx F[\Epsf]}
\leq 
   \norm{\Gradx\phi}_{L^\infty(\R^d)} 
   \vpran{|v| + \norm{v \Epsf}_{L^1(\R^{2d})}}
   + \norm{\Gradx^2 K}_{L^\infty(\R^d)}.
\end{align*}
Hence, by \eqref{bound:moment-f-k-supp} and Cauchy-Schwartz (see also \eqref{bound:F-f}),  there exists $C_1(T)$ independent of $\Eps$ such that
\begin{align} \label{bound:R-1}
   R_1 \leq C_1(T)  \, \Gfluc^{1/2}(t).
\end{align}

\smallskip
\noindent\underline{Estimate of $R_2$:} The estimate of $R_2$ is more delicate. The expression of $R_2$ and \eqref{eq:FfEps} yields
\begin{align*}
   R_2 
      &= \int_{\R^{2d}} (F[\Epsf] \cdot \del_t F[\Epsf]) \Epsf \dx\dv \\
      & \Denote R_{21} + R_{22},
\end{align*}
where 
\begin{align*}
    R_{21} = \int_{\R^{4d}} \phi(|x-y|) (v-v^\ast) \, \del_t \Epsf(t, y, v^\ast) \cdot F[\Epsf](t, x, v) \Epsf(t, x, v) \dx\dy\dv\dv^\ast
\end{align*}
and
\begin{align*}
    R_{22} = \int_{\R^{2d}} \Gradx K \ast \del_t\Epsrho (t, x) \cdot F[\Epsf](t, x, v) \Epsf(t, x, v) \dx\dv.
\end{align*}
We first estimate $R_{22}$. Integrating~\eqref{eq:f-Eps} in $v$ gives
\begin{align*}
     \del_t \Epsrho = - \Gradx \cdot \vint{v \Epsf}.
\end{align*}
Angle brackets denote integration with respect to $v$.
 
Hence by~\eqref{bound:moment-f-k-supp} and~\eqref{bound:F-f}, there exists $C_2(T)$ independent of $\Eps$ such that
\begin{align} \label{bound:R-22}
   \abs{R_{22}} 
& \leq 
    \int_{\R^{2d}} \abs{\Gradx^2 K \ast \vint{v\Epsf}} \, \abs{F[\Epsf]}(t, x, v) \Epsf(t, x, v) \dx\dv \nn
\\
& \leq  \norm{\Gradx^2 K \ast \vint{v\Epsf}}_{L^\infty(\R^d)}
           \norm{F[\Epsf] \Epsf}_{L^1(\R^{2d})}
\\
& \leq C_2(T) \, \Gfluc^{1/2}(t). \nn
\end{align}
Next we estimate $R_{21}$. Using the kinetic equation~\eqref{eq:f-Eps}, we have
\begin{align*}
    R_{21} 
    = & -\int_{\R^{4d}} \phi(|x-y|) (v-v^\ast) \vpran{v^\ast \cdot \nabla_y \Epsf(t, y, v^\ast)} \cdot F[\Epsf](t, x, v) \Epsf(t, x, v) \dx\dy\dv\dv^\ast
\\
     & + \frac{1}{\Eps}
           \int_{\R^{4d}} \phi(|x-y|) (v-v^\ast)  \nabla_{v^\ast} \cdot \vpran{F[\Epsf] \Epsf}(t, y, v^\ast) \cdot F[\Epsf](t, x, v) \Epsf(t, x, v) \dx\dy\dv\dv^\ast
\\
  & \Denote R_{21,1} + R_{21, 2} \,.
\end{align*}
Integration by parts in $R_{21,1}$ gives
\begin{align*}
    R_{21,1} = \int_{\R^{4d}} \nabla_y\phi(|x-y|) \cdot \vpran{v^\ast \Epsf(t, y, v^\ast)} (v-v^\ast)\cdot F[\Epsf](t, x, v) \Epsf(t, x, v) \dx\dy\dv\dv^\ast.
\end{align*}
Therefore, by~\eqref{bound:moment-f-k-supp} and \eqref{bound:F-f}, there exists $C_3(T)$ independent of $\Eps$ such that
\begin{align} \label{bound:R-21-1}
   \abs{R_{21,1}}
&\leq 
   \norm{\nabla \phi}_{L^\infty}
   \vpran{\norm{v\Epsf}_{L^1(\R^{2d})}
              \norm{v F[\Epsf] \Epsf}_{L^1(\R^{2d})}
              + \norm{|v|^2 \Epsf}_{L^1(\R^{2d})}
                 \norm{F[\Epsf] \Epsf}_{L^1(\R^{2d})}}
\\
&\leq C_3(T) \, \Gfluc^{1/2}(t). \nn
\end{align}
We are left to estimate $R_{21,2}$. Integration by parts in $v^\ast$ gives
\begin{align*}
   R_{21,2}
   = \frac{1}{\Eps}
           \int_{\R^{4d}} \phi(|x-y|)  \vpran{F[\Epsf] \Epsf}(t, y, v^\ast) \cdot \vpran{F[\Epsf] \Epsf}(t, x, v) \dx\dy\dv\dv^\ast \,.
\end{align*}
Combining $R_{21,2}$ with $D_0$ we get
\begin{align*}
    D_0 + R_{21,2}
    = \frac{1}{\Eps}
        \int_{\R^{4d}} \phi(|x-y|)\vpran{F^\ast- F} \cdot F  \Epsf^\ast \Epsf
        \dx\dy\dv\dv^\ast,
\end{align*}
where $F^\ast = F[\Epsf] (t, y, v^\ast)$ and $\Epsf^\ast = \Epsf(t, y, v^\ast)$. By symmetry,  \eqref{bound:phi-T} and the conservation property $\int_{\R^{2d}} F[\Epsf] \Epsf \dx\dv = 0$, we can infer further:
\begin{align} \label{bound:D-0-R-21-2}
    D_0 + R_{21,2}
&   = -\frac{1}{2\Eps}
        \int_{\R^{4d}} \phi(|x-y|) \abs{F^\ast - F}^2 \Epsf^\ast \Epsf \dx\dy\dv\dv^\ast 
\\
&  \leq     - \frac{1}{2\Eps} \tilde{C}_0(T)  
       \int_{\R^{4d}} \abs{F^\ast - F}^2 \Epsf^\ast \Epsf \dx\dy\dv\dv^\ast   \nn \\
&= - \frac{1}{\Eps} \tilde{C}_0(T)   \int_{\R^{2d}} \abs{F[\Epsf]}^2 \Epsf \dx\dv. \nn
\end{align}

Putting together \eqref{bound:R-1}, \eqref{bound:R-22}, \eqref{bound:R-21-1}, and~\eqref{bound:D-0-R-21-2} we get
\begin{align} \label{bound:D-0-R-2}
    R_1 + R_2 + D_0 
\leq 
   C_5(T) \Gfluc^{1/2}(t) - \frac{1}{\Eps} \tilde{C}_0(T) \Gfluc(t).
\end{align}
The following  differential inequality for $\Gfluc$ can now be obtained from \eqref{eqn:dG2dt} and \eqref{bound:D-0-R-2}:
\begin{align*}
   \frac{1}{2} \frac{\rm d}{{\rm d}t} \Gfluc
\leq 
   C_5(T) \Gfluc^{1/2}(t) - \frac{1}{\Eps} \tilde{C}_0(T) \Gfluc(t),
\qquad \quad
   \text{for $t \in [0, T]$}.
\end{align*}
By Gronwall's inequality, we thus have
\begin{align*}
  \Gfluc(t) 
   \leq 
      C(T) \, \Eps^2 + \Gfluc(0) \, e^{-\frac{2}{\Eps} t},
\end{align*}
for some constant $C(T)$ independent of $\Eps$.
\end{proof}


\section{Passage to the limit in the kinetic equation}
\label{sect:limit}
In this section we show the convergence of solutions $\Epsf$ of the kinetic equation~\eqref{eq:kinetic} to solutions $\Epsrho$ of the macroscopic equation~\eqref{eq:macroscopic-full}, as well as the convergence of the characteristics of~\eqref{eq:kinetic} to those of~\eqref{eq:macroscopic-full}. 

\subsection{Convergence to the macroscopic equation} 
The convergence to the macroscopic solution is given in the following theorem.

\begin{Theorem}[Convergence to the macroscopic model]\label{thm:limit}
Suppose $K$ is symmetric, $\Grad K \in
W^{1, \infty}(\Rd)$ and $\phi \in W^{1,\infty}(\mathbb{R}^+)$ is positive,
non-increasing, and satisfies the slow decay condition \eqref{cond:phi}.
For any fixed $T > 0$ and $\Eps > 0$, let $\Epsf \in
C([0, T]; \Pc(\Rd \times \Rd))$ be the unique solution
to~\eqref{eq:kinetic} with initial data $f_0 \in \Pc(\Rd \times \Rd)$ that satisfies \eqref{eqn:init-fm-zero}.
Then there exists a unique $\rho \in C([0, T]; \Pc(\R^d))$ such that
\begin{align*}
     \Epsrho(t, \cdot) \stackrel{w^*}{\longrightarrow} \rho(t ,\cdot)
\qquad
     \text{ in } \CalP(\R^d) \text { as } \Eps \to 0, \text{ for each }t \in [0, T). 
\end{align*}
Moreover, the limiting measure $\rho$ is the unique solution 
to~\eqref{eq:macroscopic-full} with initial density
$\rho_0 = \int_{\R^d} f_0 \dv$.
\end{Theorem}
\begin{proof}
We apply a similar argument as in~\cite{Jabin2000} (see \cite{FS2014} as well). The key ingredient is to study the limit involving the term $F[\Epsf]$ defined in \eqref{eq:FfEps}.  For ease of notation, denote 
\begin{equation}
\label{eq:Jeps}
J_\Eps(t,x) =  \int_{\R^d} v \Epsf(t, x, v) \dv \,.
\end{equation}
Then $F[\Epsf]$ can be written as
\begin{equation}
\label{eq:Ffeps2}
F[\Epsf] (t,x,v)
    = v \, \phi \ast \Epsrho (t,x)  - \phi \ast J_\Eps (t,x) + \nabla_x K \ast \Epsrho(t,x).
\end{equation}
We divide the proof into four steps.  \smallskip

\noindent{\underline{\em Step 1.}} \; 
First we study the passage to the limit of  $\phi \ast J_\Eps$. 
To this end, take the test function in~\eqref{eq:f-Eps-weak} as $\psi(t, x, v) = \psi_1(t) \psi_2(x) \, v\psi_3(v) $ where $\psi_1 \in C^1_c(0, T)$, $\psi_2(x), v\psi_3(v) \in C^1_b(\R^{d})$, and $\psi_3 = 1$ on $\Omega_2(T)$ (recall that $\Omega_2(T)$ is defined in~\eqref{def:supp-x-v}). 
Using this $\psi$ in~\eqref{eq:f-Eps-weak}, we have
\begin{multline*}
 \int_0^T \psi_1'(t)\int_{\R^{d}} \psi_2(x) J_\Eps(t, x) \dx \dt
    + \int_0^T \psi_1(t)\int_{\R^{2d}} \Grad\psi_2 \cdot v \otimes v \Epsf \dx\dv\dt
\\
 - \frac{1}{\Eps} 
         \int_0^T \psi_1(t) \int_{\R^{2d}} 
             \psi_2(x) F[\Epsf] \Epsf \dx\dv\dt = 0.
\end{multline*}
Therefore, if we denote
\begin{align*}
  \etaEps (t) = \int_{\R^{d}} \psi_2(x) J_\Eps(t, x) \dx,
\end{align*}
then its weak derivative is given by
\begin{equation*}
   \etaEps'(t) = \int_{\R^{2d}} \Grad\psi_2 \cdot v \otimes v \Epsf \dx\dv
                     - \frac{1}{\Eps} \int_{\R^{2d}} \psi_2(x) F[\Epsf] \Epsf \dx\dv.
\end{equation*}
By Theorem~\ref{thm:G-2}, there exists a constant $C(t_1, T)$ independent of $\Eps$ such that
\begin{align*}
    \frac{1}{\Eps} \int_{\R^{2d}} \abs{F[\Epsf]} \Epsf \dx\dv
\leq 
  \frac{1}{\Eps} \vpran{\Gfluc(t)}^{1/2} 
\leq
  C(t_1, T),
\end{align*}
for any $t \in [t_1, T]$ where $0 < t_1 < T$. Using this fact, together with Corollary~\ref{lem:moments}, we conclude that $\etaEps \in W^{1, \infty}(t_1, T)$ for any $\psi_2 \in C^1_b(\R^d)$, and
\begin{align*}
   \norm{\etaEps}_{W^{1, \infty}(t_1, T)}
\leq 
   C(t_1, T) \norm{\psi_2}_{W^{1, \infty}}
\qquad
   \text{for any $t_1 > 0$.}
\end{align*}
Hence, $\left\{ \etaEps(t) \right\}_{\Eps>0}$ is uniformly bounded and equicontinuous in $C([t_1,T])$, and by Ascoli-Arzel\`{a} Theorem it converges uniformly on a subsequence. We conclude that for any $\psi_2 \in C^1_b(\R^d)$ and $t_1 \in (0, T]$,  there exists a subsequence $\Eps_k$ and $\xi(t)$ such that
\begin{align} \label{convg:zeta}
    \int_{\R^d} \psi_2(x) J_{\Eps_k} (t, x) \dx
 \to \xi(t)
 \qquad \text{ in } C([t_1, T]) \quad  \textrm{ as } \Eps_k \to 0.
\end{align}

On the other hand, Proposition \ref{prop:unif-supp} provides a uniform (in $\Eps$) bound for the support of $\Epsf$, which in turn gives a uniform bound of $J_\Eps$. Hence, for each $t \in [0, T)$, there exists $J(t,\cdot) \in \CalM(\R^d)$ and a subsequence of $J_{\Eps_k}$, denoted as $J_{\Eps_{k_l}}(t, \cdot)$, such that 
\begin{align} \label{convg:J-0}
     J_{\Eps_{k_l}}(t,\cdot) \stackrel{w^\ast}{\longrightarrow} J(t, \cdot)
\qquad \text{ in } \CalM(\R^d) \quad  \text { as }  \Eps_{k_l} \to 0.
\end{align}
Note that in general, $\Eps_{k_l}$ could depend on $t$. However, comparing~\eqref{convg:zeta} with~\eqref{convg:J-0}, we have that along the subsequence $\Eps_k$ which is independent of $t$,  
\begin{equation}
 \label{convg:zeta-1}
    \int_{\R^d} \psi_2(x) J_{\Eps_k} (t, x) \dx
 \to \int_{\R^d} \psi_2(x) J(t, x) \dx 
 \qquad \text{ in } C([t_1, T]) \quad \text{ as } \Eps_k \to 0.
\end{equation}
By a density argument applied to $\psi_2$, we then have
\begin{equation} 
\label{convg:J-1}
     J_{\Eps_{k}}(t,\cdot) \stackrel{w^\ast}{\longrightarrow} J(t, \cdot)
\qquad \text{ in }  \CalM(\R^d) \quad  \text{ as }  \Eps_{k} \to 0,
\end{equation}
where $t \in [t_1, T]$, $t_1 > 0$ is arbitrary, and $\Eps_k$ is independent of $t$.

Furthermore, the assumption $\phi \in W^{1, \infty}(\R^d)$ guarantees that 
the sequence $\left\{ \phi \ast J_\Eps \right\}_{\Eps>0}$ is uniformly bounded and equicontinuous in $C([t_1, T] \times \overline{\Omega_1(T)})$,
where $\Omega_1(T)$ is defined in~\eqref{def:supp-x-v}. We conclude that
along a subsequence which is still denoted as $\Eps_k$,
\begin{equation}
 \label{property:equi-cont-J}
    \phi \ast J_{\Eps_k} \to \phi \ast J
\qquad
   \text{ in } C\vpran{[t_1, T] \times \overline{\Omega_1(T)}} \quad \text{ as } \Eps_k \to 0,
\end{equation}
where again, $t_1 > 0$ is arbitrary, and $\Eps_k$ is independent of $t$.
\medskip

\noindent{\underline{\em Step 2.}} \, One can use similar arguments to infer convergence results for  the other two terms entering the expression of $F[\Epsf]$ in \eqref{eq:Ffeps2}. These arguments would in fact be identical to those used in \cite{FS2014} for the attractive-repulsive model with no alignment. We summarize the key ideas briefly and refer to the proof of Theorem 5.1 in \cite{FS2014} for details.

By taking a test function of the form $\psi(t, x, v) = \psi_1(t) \psi_2(x) $ in \eqref{eq:f-Eps-weak}, one can show, by similar calculations as in Step 1, that $\int_{\Rd} \psi_2(x) \Epsrho(t,x)$ is uniformly bounded in $W^{1,\infty}(0,T)$. Together with the tightness of the sequence $\Epsrho(t,\cdot)$ and the application of Prokhorov's theorem (cf. \cite[Theorem 4.1]{Patrick1971}), it can be inferred that there exists a probability measure $\rho(t,\cdot) \in \CalP(\R^d)$ such that, along a subsequence of $\Eps_k$ (still denoted as $\Eps_k$), 
\begin{equation} 
\label{eq:conv-Epsrhok}
     \rho_{\Eps_k} (t, \cdot)  \stackrel{w^\ast}{\longrightarrow} \rho(t, \cdot)
\qquad \text{ in } \CalP(\R^d) \quad \text { as } \Eps_{k} \to 0.
\end{equation}
The convergence above holds for each $t \in (0,T]$ and the sequence $\Eps_k$ is independent of $t$.

Also as in Step 1, one can get
\begin{equation}  
\label{property:equi-cont-rho}
    \phi \ast \rho_{\Eps_k} \to \phi \ast \rho
\qquad 
   \text{ in } C\vpran{[t_1, T] \times \overline{\Omega_1(T)}} \quad \text{ as } \Eps_k \to 0,
\end{equation}
where the uniformity with respect to $x$ follows from the equicontinuity in $x$ of $\{\phi \ast \rho_{\Eps}(t,\cdot) \}$.

Finally, regarding the convergence of $\{ \Grad K \ast \rho_{\Eps_k} \} $, the equicontinuity in $x$ is immediate (by the assumption made on the potential). The equicontinuity with respect to time  is slightly more delicate, as $\Grad K$ does not have enough regularity to be used as a test function. Nevertheless, one can overcome this difficulty by a regularization of the kernel (as shown in \cite{FS2014}) and it can  be inferred that
\begin{equation}  
\label{property:equi-gradK-rho}
    \Grad K \ast \rho_{\Eps_k} \to \Grad K \ast \rho
\qquad 
   \text{ in } C\vpran{[t_1, T] \times \overline{\Omega_1(T)}} \quad \text{ as } \Eps_k \to 0,
\end{equation}
with $t_1>0$ arbitrary.
\medskip

\noindent{\underline{\em Step 3.}} \, We now proceed to deriving the limiting equation~\eqref{eq:u-implicit-full} at each $t \in [0, T)$.  For all $(x, v) \in \Omega(T)$, we use \eqref{eq:Ffeps2}, \eqref{property:equi-cont-J}, \eqref{property:equi-cont-rho} and~\eqref{property:equi-gradK-rho} to derive that
\begin{align*}
    F[f_{\Eps_k}](t, \cdot, \cdot) \to F_0(t, \cdot, \cdot)
\qquad 
    \text{ strongly in } C_b(\Omega(T)) \quad \text{ for each } t \in (0,T],
\end{align*}
where 
\begin{align*}
    F_0 (t,x,v )= v \, \phi \ast \rho(t,x) - \phi \ast J(t,x) + \Grad K \ast \rho (t,x).
\end{align*}
Therefore, for each time $t \in (0,T]$, 
\begin{equation*}
    \int_{\R^{2d}} \vpran{F[f_{\Eps_k}] - F_0}^2 f_{\Eps_k} \dx\dv  
\leq \norm{F[f_{\Eps_k}] - F_0}_{C_b(\Omega_T)}^2 \to  0  \qquad  \text{ as } \Eps_k \to 0.
\end{equation*}
Combining the above convergence with the vanishing in $\Eps$ of $\Gfluc(t)$, as established in Theorem~\ref{thm:G-2}, we have 
\begin{equation*}
    \int_{\R^{2d}} F_0^2 \, f_{\Eps_k} \dx \dv \to 0
\qquad
   \text{ as } \Eps_k \to 0.
\end{equation*}
Let $\varphi_1 \in C^1_b(\R^d)$ be arbitrary. Then, by Cauchy-Schwarz inequality, 
\begin{equation*}
    \int_{\R^{2d}} \varphi_1(x) F_0 f_{\Eps_k} \dv \dx \to 0 \qquad  \text{ as } \Eps_k \to 0,
\end{equation*}
and hence,
\begin{equation*}
    \int_{\R^d} \varphi_1(x) \vpran{\Phi J_{\Eps_k} - \rho_{\Eps_k} \vpran{\phi \ast J + \Grad K \ast \rho}} \dx \to  0 \qquad  \text{ as } \Eps_k \to 0.
\end{equation*}
where $\Phi = \phi \ast \rho \in C^1_b(\R^d)$. 
By~\eqref{convg:J-1} and~\eqref{eq:conv-Epsrhok} we can then pass to the limit and get 
\begin{align*}
    \int_{\R^d} \varphi_1(x) \vpran{\Phi J - \rho \vpran{\phi \ast J + \Grad K \ast \rho}} \dx = 0
 \quad
    \text{for any $\varphi_1 \in C^1_b(\Rd)$.}
\end{align*}
This implies that for each $t \in (0, T]$,
\begin{equation}
\label{eqn:phiJ}
    \Phi(t,x) J(t, x) 
 = \rho(t, x) \vpran{\int_{\R^{d}} \phi(|x-y|) J(t,y)\dy   - \Grad K \ast \rho(t,x)}.
\end{equation}
By denoting
\begin{equation}
\label{eqn:defn-uJ}
u(t,x) = J(t,x)/\rho(t,x), \qquad \text{ on the support of } \rho(t,\cdot),
\end{equation}
then equation~\eqref{eq:u-implicit-full} holds on $\Supp\rho$. 

\medskip

\noindent{\underline{\em Step 4.}} \, In this last step we show that the limit $\rho$ satisfies the macroscopic equation  \eqref{eq:rho-full}, as well as the momentum conservation \eqref{eq:conserv-momentum}.  First, following arguments used in \cite{FS2014}, one can generalize \eqref{convg:zeta-1} and show that $\psi_2$ can be allowed to depend on $t$ as well. Specifically, the following convergence holds:
\begin{equation} 
\label{convg:general-w-t-1}
    \int_{\R^d} \psi_2(t,x) J_{\Eps_k} (t, x) \dx
 \to \int_{\R^d} \psi_2(t,x) J (t, x) \dx
 \qquad \text{ in } C([t_1, T]) \quad  \text{ as } \Eps_k \to 0,
\end{equation}
for all $\psi_2 \in C_c([t_1,T),C^1_b(\Rd))$.

Similarly, an analogous result can be shown for the convergence of $\rho_{\Eps_k}$:
\begin{equation}
\label{convg:general-w-t-2}
    \int_{\R^d} \psi_2(t,x) \rho_{\Eps_k} (t, x) \dx
 \to \int_{\R^d} \psi_2(t,x) \rho (t, x) \dx
\qquad 
   \text{ in } C\vpran{[t_1, T]} \quad  \text{ as } \Eps_k \to 0,
\end{equation}
for all $\psi_2 \in C_c([t_1,T),C^1_b(\Rd))$.

Now, in the weak formulation \eqref{eq:f-Eps-weak} for $f_{\Eps_k}$, take a function $\psi$ of the form 
$\psi(t, x, v) = \varphi(t, x) \in C^1_c([0, T), C^1_b(\R^d ))$ to get
\begin{equation} 
\label{eq:f-Eps-weak-1}
\begin{aligned}
& \int_{0}^T \int_{\R^{d}} \del_t \varphi(t, x) \rho_{\Eps_k}(t, x) \dx\dt
    + \int_{0}^T \int_{\R^d} \Grad \varphi\cdot J_{\Eps_k} \dx\dt
    + \int_{\R^d} \varphi(0, x) \rho_0(x, v) \dx=0,
\end{aligned}
\end{equation}
where $\rho_0 = \int_{\R^d} f_0(x, v) \dv$.  

Fix $\Eps_1>0$ small. Then, one can choose $t_1>0$ small enough such that
\begin{equation}
\label{eqn:ineq-t1}
\int_0^{t_1} \int_{\R^d} \abs{\del_t \varphi(t, x) \rho_{\Eps_k}(t, x)}  \dx\dt
+ \int_0^{t_1} \int_{\R^d} \abs{\Grad \varphi \cdot J_{\Eps_k}} \dx\dt
\leq 
 C \norm{\varphi}_{W^{1,\infty}_{t,x}} t_1
<
  \Eps_1.
\end{equation}

With $t_1$ fixed, such that \eqref{eqn:ineq-t1}  is satisfied, break the time integrals $\int_0^T$ in \eqref{eq:f-Eps-weak-1} into two pieces: $\int_0^{t_1}$ and $\int_{t_1}^T$. Due to \eqref{convg:general-w-t-1} and \eqref{convg:general-w-t-2} one can pass to the limit $\Eps_k\to 0$ in the time integrals $\int_{t_1}^T$. Hence, we infer from \eqref{eq:f-Eps-weak-1} and \eqref{eqn:ineq-t1} that for every $\Eps_1>0$, there exists $0<t_1<\Eps_1/ C \norm{\varphi}_{W^{1,\infty}_{t,x}}$ such that
\begin{equation*} 
\left| \int_{t_1}^T \int_{\R^{d}} \del_t \varphi(t, x) \rho(t, x) \dx\dt
    + \int_{t_1}^T \int_{\R^d} \Grad \varphi \cdot J \dx\dt
    + \int_{\R^d} \varphi(0, x) \rho_0(x, v) \dx\right|  \leq \Eps_1.
\end{equation*}
We conclude
\begin{equation} 
\label{eq:f-Eps-weak-2}
\begin{aligned}
& \int_0^T \int_{\R^{d}} \del_t \varphi(t, x) \rho(t, x) \dx\dt
    + \int_0^T \int_{\R^d} \Grad \varphi \cdot J \dx\dt
    + \int_{\R^d} \varphi(0, x) \rho_0(x, v) \dx=0,
\end{aligned} 
\end{equation}
for all  $\varphi(t, x) \in C^1_c([0, T), C^1_b(\R^d ))$.
Note that \eqref{eq:f-Eps-weak-2} is equivalent to \eqref{eq:rho-weak} since  $J=\rho u$ on the support of $\rho$ and $J = 0$ outside the support of $\rho$.

The conservation of momentum \eqref{eq:conserv-momentum} holds since it holds for every $J_{\Eps_k}$. 
Indeed, by \eqref{eqn:feps-zeromom}, we then have 
\begin{align*}
     \int_{\R^{d}} J_{\Eps_k}(t, x) \dx = 0
\qquad
    \text{for all } t\in[0,T] \text { and } \Eps_k>0.
\end{align*}
Consequently, by \eqref{convg:zeta-1} and \eqref{eqn:defn-uJ}, we infer that \eqref{eq:conserv-momentum} is satisfied.

We conclude by noting that the convergence above was derived on a subsequence $\Eps_k$. However, by the uniqueness of the weak solution to~\eqref{eq:macroscopic-full}, as established in Theorem \ref{thm:macro}, the {\em full} sequence $\Epsrho(t, \cdot)$ converges to $\rho(t, \cdot)$ for each $t \in [0, T]$.
\end{proof}


\subsection{Convergence of characteristic paths} In this part we investigate in more detail the convergence of $\Epsf$. In particular, we show that the characteristic paths of~\eqref{eq:kinetic} converge to those of~\eqref{eq:macroscopic-full}. This provides a geometric point of view of the singular limiting process, which also provides an explicit formula for the limit of $\Epsf$. 

The characteristic system \eqref{eq:characteristics} along which the solution $\Epsf$ is transported has the form 
\begin{align} \label{eq:characteristics2}
    &\frac{{\rm d} x}{\dt} = v, \nn
\\
   &\Eps \frac{{\rm d} v}{\dt} 
     = - \phi \ast \Epsrho (t,x) \; v + \phi \ast J_\Eps(t,x)  - \Grad K \ast \rho_\Eps(t, x),
\\
  & (x, v) \big|_{t=0} = (x_0, v_0) \in {\rm supp}f_0. \nn   
\end{align}
The goal here is to pass the limit $\Eps \to 0$ in the characteristic system \eqref{eq:characteristics2} and relate the limit to the characteristic paths of the limiting macroscopic equation \eqref{eq:macroscopic-full}. 

The main tool in showing this limit is a classical result  in singular perturbation theory due to Tikhonov \cite{Tikhonov1952}. A brief account of this result is the following. Using notations from \cite{Vasileva1963}, consider the general system 
\begin{equation}
\label{eqn:shorthand}
\left\{
  \begin{array}{l}
    \dfrac{dx}{dt} = v,\\\\
    \Eps\dfrac{dv}{dt} = \CalF(x,v,t),
  \end{array}
\right.
\end{equation}
where $x,v \in \mathbb{R}^{d}$ and $\Eps>0$.

System~\eqref{eqn:shorthand} is a two-scale equation, where $t$ and $\tau = t/\Eps$ represent the slow and fast time scales, respectively.  The slow dynamics of \eqref{eqn:shorthand} is given by 
\begin{equation}
\label{eqn:shorthand:fo}
\left\{
  \begin{array}{l}
    \dfrac{dx}{dt} = v,\\\\
    v = \Gamma(x,t),
  \end{array}
\right.
\end{equation}
where $v=\Gamma(x,t)$ is a root of the equation
\begin{equation}
\label{eqn:calFroots}
\CalF(x,v,t)=0.
\end{equation}
Roots $v = \Gamma(x,t)$ of \eqref{eqn:calFroots} are in general non-unique. For a fixed configuration $x^\ast$ and  time $t^\ast$, the fast dynamics is defined by
\begin{equation}\label{eqn:adj}
\dfrac{dv}{d\tau}=\CalF(x^\ast,v,t^\ast).
\end{equation}
Note that in the fast system \eqref{eqn:adj}, $x^\ast$ and $t^\ast$ are being regarded as parameters.

Tikhonov's theorem establishes conditions on a root $v = \Gamma(x,t)$ which guarantee that solutions of the two-scale system \eqref{eqn:shorthand} converge, via a fast initial layer governed by \eqref{eqn:adj}, to solutions of the degenerate/slow system \eqref{eqn:shorthand:fo}.

A root $v=\Gamma(x,t)$, with $\Gamma$ defined on  a closed and bounded set $D \subset \R^{d+1}$, is
called \textit{isolated} if there is a $\delta>0$ such that for all $(x,t)\in D$, the only element in $B(\Gamma(x,t),\delta)$ that satisfies $\CalF(x,v,t)=0$ is $v=\Gamma(x,t)$. An isolated root $\Gamma$ is called \textit{positively stable} in $D$, if $v^\ast=\Gamma(x^\ast,t^\ast)$ is an asymptotically stable stationary point of \eqref{eqn:adj} as $\tau\to\infty$, for each $(x^\ast,t^\ast)\in D$. The \textit{domain of influence} of an isolated positively stable root $\Gamma$ is the set of points $(x^\ast,\tilde{v},t^\ast)$ such that the solution of \eqref{eqn:adj} satisfying $v|_{\tau=0}=\tilde{v}$ tends to $v^\ast = \Gamma(x^\ast,t^\ast)$ as $\tau\to\infty$.

Tikhonov's theorem  \cite{Tikhonov1952} states the following:
\begin{Theorem}[Tikhonov \cite{Tikhonov1952, Vasileva1963}] \label{thm:Tikhonov}
Assume that a root $v = \Gamma(x,t)$ of \eqref{eqn:calFroots} is isolated positively stable in some bounded closed domain $D$. Consider a point $(x_0,v_0,t_0)$ in the domain of influence of this root, and assume that the slow equation  \eqref{eqn:shorthand:fo} has a solution
 $x(t)$ initialized at $x(t_0) = x_0$, such that $(x(t),t)$ lies in $D$ for all $t\in[t_0,T]$. Then, as $\Eps\to 0$, the solution 
$(x_\Eps (t),v_\Eps (t))$ 
of \eqref{eqn:shorthand} initialized at $(x_0,v_0)$, converges to $(x(t),v(t)):=(x(t),\Gamma(x(t),t))$ in the following sense:

i) $\displaystyle\lim_{\Eps\to 0}v_\Eps(t)=v(t)$ \text{ for all } $t\in(t_0,T^*]$, \text {and }
\smallskip

ii) $\displaystyle\lim_{\Eps\to 0}x_\Eps(t)=x(t)$ \text{ for all } $t\in[t_0,T^*]$,

\smallskip
for some $T^*<T$.
\end{Theorem}

\begin{rmk}\label{remark:boundary layer}
The convergence of $v_\Eps(t)$ to $v(t)$ occurs via a fast initial layer and normally does not occur at the initial time $t_0$, unless the initial data satisfies $v_0=\Gamma(x_0,t_0)$.
\end{rmk}

Our main result concerning convergence of characteristic paths is the following theorem.
\begin{Theorem}[Convergence of characteristic paths]
\label{thm:conv-traj} Assume $K$ and $\phi$ satisfy the same assumptions as in Theorem \ref{thm:limit}. Consider the measure-valued solution $\Epsf$ to \eqref{eq:kinetic} and the characteristic path $(x_\Eps(t),v_\Eps(t))$ that originates  from some $(x_0,v_0) \in \operatorname{supp} f_0$ at $t=0$. Then,
\begin{equation}
\label{eqn:conv-trajx}
\lim_{\Eps \to 0} (x_\Eps(t), v_\Eps(t)) 
= (x(t), u(t, x(t)))  \quad \text{ for all } \quad 0\leq t \leq T,
\end{equation}
where $x(t)$ is the characteristic trajectory of the limiting macroscopic equation \eqref{eq:macroscopic-full} that starts at $x_0$, and $u(t,x)$ is the velocity field defined in \eqref{eqn:defn-uJ}. In particular, $x(t)$ satisfies \eqref{eq:char-limit}, and $u$ solves \eqref{eq:u-implicit-full}.
\end{Theorem}
\begin{proof}
The characteristic system \eqref{eq:characteristics2} has a right-hand-side that depends on $\Eps$ and Tikhonov's theorem does not apply directly. To circumvent this, we replace $(\Epsrho, J_\Eps)$ by $(\rho, J)$ 
in  \eqref{eq:characteristics2} to arrive at the following system:
 \begin{align} \label{eq:char-noeps}
    &\frac{{\rm d} x}{\dt} = v, \nn
\\
   &\Eps \frac{{\rm d} v}{\dt} 
     = - \phi \ast \rho (t,x) \; v + \phi \ast J(t,x)  - \Grad K \ast \rho(t, x),
\\
  & (x, v) \big|_{t=0} = (x_0, v_0) \in {\rm supp}f_0. \nn
\end{align}

Theorem \ref{thm:Tikhonov} applies to system \eqref{eq:char-noeps}. Indeed, \eqref{eq:char-noeps} can be written in the form \eqref{eqn:shorthand}, with 
\[
\CalF(x,v,t) =  - \phi \ast \rho (t,x) \; v + \phi \ast J(t,x)  - \Grad K \ast \rho(t, x).
\]
For a fixed spatial configuration $x$ and time $t$, the root $v= \Gamma(x,t)$ given by 
\begin{equation}
\label{eqn:Gamma-u}
\Gamma(x,t) = \frac{1}{ \phi \ast \rho (t,x)} (\phi \ast J(t,x)  - \Grad K \ast \rho(t, x))
\end{equation}
is unique, hence isolated. It is also immediate that for a fixed spatial configuration $x^\ast$ and time $t^\ast$, the corresponding fast equation \eqref{eqn:adj} has a globally attracting equilibrium $v^\ast = \Gamma(x^\ast,t^\ast)$. Consequently, $v^\ast$ is positively stable and its domain of influence is $\{x^\ast\} \times \Rd \times \{t^\ast \}$.

Denote by $(\tilde{x}_\Eps(t),\tilde{v}_\Eps(t))$ the solution of
\eqref{eq:char-noeps} that originates from $(x_0,v_0)$. Then, by Theorem
\ref{thm:Tikhonov}, the convergence in \eqref{eqn:conv-trajx}, which needs to be shown for ${x}_\Eps(t)$ and ${v}_\Eps(t)$, holds for $\tilde{x}_\Eps(t)$ and $\tilde{v}_\Eps(t)$ (note that by \eqref{eqn:phiJ}, the root $\Gamma$ in \eqref{eqn:Gamma-u} is in fact the velocity $u$ defined by \eqref{eqn:defn-uJ}). Hence, it would be enough to show that for a fixed $t>0$,
\begin{equation}
\label{conv:xv}
\lim_{\Eps \to 0} |x_\Eps(t)-\tilde{x}_\Eps(t)| = 0 \quad \text{ and } \quad \lim_{\Eps \to 0} |v_\Eps(t)-\tilde{v}_\Eps(t)| = 0.
\end{equation}
Indeed, from \eqref{eq:characteristics2} and \eqref{eq:char-noeps} we get
\begin{align*}
\Eps \frac{{\rm d}}{\dt} (v_\Eps(t)-\tilde{v}_\Eps(t)) 
&= - \phi \ast \Epsrho (t,x_\Eps(t)) v_\Eps(t) + \phi \ast \rho (t,\tilde{x}_\Eps(t))\tilde{v}_\Eps(t)  
\\
&\hspace{-1.0cm}
+ \phi \ast J_\Eps (t,x_\Eps(t))- \phi\ast J(t,\tilde{x}_\Eps(t)) -\nabla_x K \ast \Epsrho (t,x_\Eps(t)) + \nabla_x K \ast \rho (t,\tilde{x}_\Eps(t)).
\end{align*}
To the first line on the right-hand-side we add and subtract $\phi \ast \Epsrho (t,x_\Eps(t)) \tilde{v}_\Eps(t)$. Then, we can write
\begin{equation}
\label{eq:dtdiff1}
\Eps \frac{{\rm d}}{\dt} (v_\Eps(t)-\tilde{v}_\Eps(t)) = - \phi \ast \Epsrho (t,x_\Eps(t)) \left ( v_\Eps(t) - \tilde{v}_\Eps(t) \right) + \Rem (t),
\end{equation}
where
\begin{align}
\label{eq:calG}
\Rem (t) &= \left( \phi \ast \Epsrho (t,x_\Eps(t))- \phi \ast \rho (t,\tilde{x}_\Eps(t)) \right)  \tilde{v}_\Eps(t) \\
&+  \phi \ast J_\Eps (t,x_\Eps(t))- \phi\ast J(t,\tilde{x}_\Eps(t)) -\nabla_x K \ast \Epsrho (t,x_\Eps(t)) + \nabla_x K \ast \rho (t,\tilde{x}_\Eps(t)). \nonumber
\end{align}
By integrating \eqref{eq:dtdiff1} one finds
\begin{equation}
\label{eq:dtdiff2}
v_\Eps(t)-\tilde{v}_\Eps(t) = \frac{1}{\Eps} \int_0^t e^{-\frac{1}{\Eps} \int_s^t  \phi \ast \Epsrho (\tau,x_\Eps(\tau)) d \tau} \,\Rem(s) ds, \qquad \text{ for all } t \in [0,T].
\end{equation}
Using Proposition \ref{prop:unif-supp}, we have
\[
\phi \ast \Epsrho (\tau,x_\Eps(\tau)) \geq C_1, \qquad \text{ for all } \tau \in [0,T],
\]
where the constant $C_1$ depends on $T$, $\phi$ and $K$, but not on $\Eps$. Consequently, from \eqref{eq:dtdiff2}, we get
\begin{equation}
\label{eq:dtdiff3}
| v_\Eps(t)-\tilde{v}_\Eps(t) | \leq  \frac{1}{\Eps} \int_0^t e^{-\frac{C_1}{\Eps} (t-s)} \, | \Rem(s)| ds, \qquad \text{ for all } t \in [0,T].
\end{equation}

We now focus on estimating the right-hand-side of \eqref{eq:dtdiff3}. Inspect $\Rem$ given by \eqref{eq:calG}, in particular the term on the first line of the right-hand-side. The term $\tilde{v}_\Eps(t)$ is uniformly bounded in $\Eps$ (to show this, one can follow for instance the arguments used to prove Proposition \ref{prop:unif-supp}). To estimate the term in round brackets, add and subtract $\phi \ast \rho (t,x_\Eps(t))$. Then, by triangle inequality and the Mean Value Theorem, we get
\begin{align*}
|\phi \ast \Epsrho (t,x_\Eps(t))- \phi \ast \rho (t,\tilde{x}_\Eps(t))| & \leq 
\sup_{t\in[0,T]} \| \phi \ast (\Epsrho - \rho)\|_{L^\infty(\Omega_1(T))} \\
&+ \sup_{t\in[0,T]} \| \nabla \phi \ast \rho \|_{L^\infty(\Omega_1(T))} |x_\Eps(t)-\tilde{x}_\Eps(t)|
\end{align*}
Estimates entirely similar to the one above can be made for the terms on the second line of the right-hand-side of \eqref{eq:calG}, which lead to:
\begin{align*}
|\phi \ast J_\Eps (t,x_\Eps(t))- \phi \ast J (t,\tilde{x}_\Eps(t))| & \leq 
\sup_{t\in[0,T]} \| \phi \ast (J_\Eps - J)\|_{L^\infty(\Omega_1(T))} \\
&+ \sup_{t\in[0,T]} \| \nabla \phi \ast J \|_{L^\infty(\Omega_1(T))} |x_\Eps(t)-\tilde{x}_\Eps(t)|,
\end{align*}
and
\begin{align*}
|\nabla_x K \ast \Epsrho (t,x_\Eps(t))- \nabla_x K \ast \rho (t,\tilde{x}_\Eps(t))| & \leq 
\sup_{t\in[0,T]} \| \nabla_x K \ast (\Epsrho - \rho)\|_{L^\infty(\Omega_1(T))} \\
&+ \sup_{t\in[0,T]} \| \nabla^2_x K \ast \rho \|_{L^\infty(\Omega_1(T))} |x_\Eps(t)-\tilde{x}_\Eps(t)|.
\end{align*}
where $\Omega_1(T)$ is defined in~\eqref{def:supp-x-v}.

By the uniform convergences of $\phi \ast \Epsrho$, $\phi \ast J_\Eps$, and $\nabla_x K \ast \Epsrho$, as established in the proof of Theorem \ref{thm:limit}, and by assumptions we made on $\phi$ and $K$, we can group the three estimates from above and get from \eqref{eq:calG},
\begin{equation}
\label{eqn:estG}
|\Rem(t) | \leq C_\Eps + C_2 |x_\Eps(t) - \tilde{x}_\Eps(t)|,
\end{equation}
where $C_\Eps$ and $C_2$ are constants that  depends on  $T$, $\phi$ and $K$. Moreover, $C_\Eps \to 0$ as $\Eps \to 0$.

Applying \eqref{eqn:estG} in \eqref{eq:dtdiff3}, we get
\begin{equation*}
| v_\Eps(t)-\tilde{v}_\Eps(t) | \leq \frac{C_\Eps}{C_1} +  \frac{C_2}{\Eps} \int_0^t e^{-\frac{C_1}{\Eps} (t-s)}  |x_\Eps(s) - \tilde{x}_\Eps(s)| ds, \qquad \text{ for all } t \in [0,T].
\end{equation*}

Using \eqref{eq:characteristics2} and \eqref{eq:char-noeps}, we further find
\begin{equation}
\label{eq:dtdiff4}
| v_\Eps(t)-\tilde{v}_\Eps(t) | \leq \frac{C_\Eps}{C_1} +  \frac{C_2}{\Eps} \int_0^t e^{-\frac{C_1}{\Eps} (t-s)}  \int_0^s |v_\Eps(\tau) - \tilde{v}_\Eps(\tau)| \, d\tau \, ds, \qquad \text{ for all } t \in [0,T].
\end{equation}
Change order of integration in the double integral on the right-hand-side of \eqref{eq:dtdiff4} to get
\begin{equation*}
| v_\Eps(t)-\tilde{v}_\Eps(t) | \leq \frac{C_\Eps}{C_1} +  \frac{C_2}{\Eps} \int_0^t  |v_\Eps(\tau) - \tilde{v}_\Eps(\tau)| \int_{\tau}^t e^{-\frac{C_1}{\Eps} (t-s)}  \, ds \, d\tau, \qquad \text{ for all } t \in [0,T].
\end{equation*}
After evaluating the integral in $s$ we finally arrive at
\begin{equation*}
| v_\Eps(t)-\tilde{v}_\Eps(t) | \leq \frac{C_\Eps}{C_1} +  \frac{C_2}{C_1} \int_0^t  |v_\Eps(\tau) - \tilde{v}_\Eps(\tau)|  d\tau, \qquad \text{ for all } t \in [0,T].
\end{equation*}
The second limit in \eqref{conv:xv} now follows from the integral form of Gronwall's inequality, given that $C_\Eps \to 0$ as $\Eps \to 0$. The convergence of trajectories follows from here as well.
\end{proof}

The convergence of characteristic paths yields the limiting flow map $\mathcal{T}^t: x_0 \to x(t)$. 
It is convenient in the calculations below to use the notation $x(t;x_0)$ to denote the limiting characteristic path $x(t)$ that starts at $x_0$. 

The next result characterizes the limiting densities.


\begin{Theorem}[Characterization of the limiting densities] 
\label{thm:rho-trans}
The limiting macroscopic density $\rho$ identified in Theorem \ref{thm:limit} is the push-forward of the initial density $\rho_0$ by the limiting flow map $\mathcal{T}^t$, 
\begin{equation}
\label{eqn:rho-trans}
\rho = \mathcal{T}^t \# \rho_0.
\end{equation}
In addition,  for each $t\in [0,T)$, $\Epsf$ converges weak-$^\ast$ to a probability density $f(t,\cdot,\cdot)\in \CalP(\Rd \times \Rd)$:
\begin{equation}
\label{eqn:fconv}
   \Epsf \stackrel{w^\ast}{\longrightarrow} f
\qquad \text{ in } \CalP(\Rd \times \Rd) \quad \text{ as } \Eps \to 0.
\end{equation}
The limiting density $f$, with first marginal $\rho$, is given explicitly by: 
\begin{equation}
\label{eqn:f-delta}
f(t,x,v) = \rho(t,x) \delta (v- u(t,x)),
\end{equation}
where $(\rho,u)$ is the unique solution of \eqref{eq:macroscopic-full}.
\end{Theorem}
\begin{proof}
The first part, expressed by equation \eqref{eqn:rho-trans}, follows from considerations made in Theorem \ref{thm:limit}. However, it also follows directly, as a consequence of the argument below.

The limiting behaviour of $\Epsf$ was not explicitly stated or needed in Theorem \ref{thm:limit}, but follows by arguments similar to those used for $J_\Eps$ and $\Epsrho$ in the proof of Theorem \ref{thm:limit}. Let us sketch this argument briefly.

Fix $\psi_1 \in C_c^1(0, T)$ and $\psi_2 \in C^1_b(\Rd \times \Rd)$, and let $\varphi(t, x, v) = \psi_1(t)\psi_2(x,v)$ in \eqref{eq:f-Eps-weak}. Find
\begin{multline}
\label{eqn:feps-weak-1}
    \int_0^T \psi'_1(t)\int_{\R^{2d}} \psi_2(x,v) \Epsf(t, x,v) \dx \dv \dt
    =\\
  - \int_0^T \psi_1(t) \int_{\R^{2d}}  \left[ \nabla_x \psi_2 \cdot v  - \frac{1}{\Eps} \nabla_v \psi_2 \cdot F[\Epsf] \right] \Epsf \dx\dv\dt.
\end{multline}
Denoting by
\begin{equation*}
\tetaEps(t) = \int_{\R^{2d}} \psi_2(x,v) \Epsf(t, x,v) \dx \dv,
\end{equation*}
then, by \eqref{eqn:feps-weak-1}, the weak derivative of $\tetaEps$ is given by
\begin{equation*}    
 {\tetaEps}^{\,^{\prime}}(t)
    = \int_{\R^{2d}} \left[ \nabla_x \psi_2 \cdot v  - \frac{1}{\Eps} \nabla_v \psi_2 \cdot F[\Epsf] \right]   \Epsf \dx\dv 
    \in L^\infty(0, T).
\end{equation*}
The right-hand-side of the equation above is bounded (the boundedness of the term that contains $\Eps$ follows from Theorem \ref{thm:G-2}).

Since $\tetaEps$ is uniformly bounded in $W^{1, \infty}(0, T)$, it converges uniformly on a subsequence. On the other hand, similar to the arguments used for $J_\Eps$ and $\Epsrho$ in the proof of Theorem \ref{thm:limit}, we note that the sequence  $\Epsf(t,\cdot,\cdot) \in \CalP(\Rd \times \Rd) $  is tight,  and hence, for each $t \in [0, T)$, $\Epsf(t,\cdot,\cdot)$ converges weak-$^\ast$ as measures, on a subsequence $\Eps_k$,  
to a probability measure $f(t, \cdot,\cdot) \in \CalP(\R^d \times \Rd)$. Note also that the subsequence $\Eps_k$ does not depend on $t$, due to the equicontinuity of $\tetaEps(t)$ derived above.

Hence,
\begin{equation}
 \label{convf:6}
     f_{\Eps_k} (t, \cdot,\cdot)  \stackrel{w^\ast}{\longrightarrow} f(t, \cdot, \cdot)
\qquad \text{ in } \CalP(\Rd \times \Rd) \quad \text{ as } \Eps_k \to 0,
\end{equation}
which proves the convergence \eqref{eqn:fconv} on a subsequence. To show its  convergence on the full sequence $\Eps\to 0$ we use the uniqueness of $f$, as derived from the arguments below.

Since $\Epsf(t) = \mathcal{T}^{t,\Eps}_{\mathcal{H}[\Epsf]} \# f_0$ by Definition \ref{defn:sol}, \eqref{eqn:trans1} holds for $f_{\Eps_k}$:
\begin{equation}
\label{eqn:trans2}
\int_{\R^{2d}} \zeta (x,v) f_{\Eps_k}(t,x,v) dx dv = \int_{\R^{2d}} \zeta(\mathcal{T}^{t,\Eps_k}_{\mathcal{H}[f_{\Eps_k}]}(X,V)) f_0(X,V) dX dV,
\end{equation}
for all $\zeta \in C_b(\Rd \times \Rd)$.

By the weak-$^\ast$ convergence of $f_{\Eps_k}$, the left-hand-side of \eqref{eqn:trans2} converges as $\Eps_k \to 0$:
\[
\int_{\R^{2d}} \zeta (x,v) f_{\Eps_k}(t,x,v) dx dv \to \int_{\R^{2d}} \zeta (x,v) f(t,x,v) dx dv.
\]
Due to convergence of trajectories \eqref{eqn:conv-trajx}, 
the right-hand-side of \eqref{eqn:trans2} converges by Lebesgue's dominated convergence theorem,
\[
\int_{\R^{2d}} \zeta(\mathcal{T}^{t,\Eps_k}_{\mathcal{H}[f_{\Eps_k}]}(X,V)) f_0(X,V) dX dV \to \int_{\R^{2d}} \zeta(x(t;X),
u(t,x(t;X))) f_0(X,V) dX dV,
\]
as $\Eps_k \to 0$. Combining these two, we find
\begin{equation}
\label{eqn:equal1}
 \int_{\R^{2d}} \zeta (x,v) f(t,x,v) dx dv = \int_{\R^{2d}} \zeta(x(t;X), u(t,x(t;X))) f_0(X,V) dX dV,
\end{equation}
for all $\zeta \in C_b(\Rd \times \Rd)$.

First note that \eqref{eqn:rho-trans} can be derived from \eqref{eqn:equal1}. Indeed, choose $\zeta(x,v) = \varphi(x)$ in \eqref{eqn:equal1} to find 
\[
\int_{\Rd} \varphi(x) \rho(t,x) dx = \int_{\Rd} \varphi (x(t;X)) \rho_0(X) dX,
\]
for all $\varphi \in C_b(\Rd)$; this represents exactly the mass transport given by \eqref{eqn:rho-trans}.

Now, observe that \eqref{eqn:f-delta} is equivalent to 
\begin{equation*}
 \int_{\R^{2d}} f(t,x,v) \zeta(x,v) dx dv = \int_{\Rd}  \zeta(x,u(t,x)) \rho(t,x) dx,
\end{equation*}
for all test functions $\zeta \in C_b(\Rd \times \Rd)$, which can be inferred immediately from \eqref{eqn:rho-trans} and \eqref{eqn:equal1}.

The {\em unique} explicit representation of the limiting density $f$ implies that the convergence in \eqref{convf:6} holds on the full sequence $\Epsf$, as stated in \eqref{eqn:fconv}.
\end{proof}


\section{Numerical implementation and large time behaviour}
\label{sect:numerics}

In this section, we present and implement a numerical scheme for the
macroscopic system ~\eqref{eq:macroscopic-full}. The numerical approach taken here follows closely the analytical considerations made in Section \ref{sect:well-posedness}. In particular, it provides a discrete analogue for replacing the singular operator $\CalA$ by an invertible operator $\CalM$.

\subsection{Discrete setting} 
\label{subsect:discr-setting}
The numerical discretization of the evolution equation \eqref{eq:rho-full} for the density $\rho$  can be done by standard methods. For instance, spatial discretization can be performed by a finite volume method \cite{CCH2015} or by a semi-Lagrangian particle method \cite{DoHuPoRo2004}. For this reason, we focus in this section on the numerical solution of the velocity equation \eqref{eq:u-implicit-full}.  
The subtlety is that for a given $\rho$, this equation does not have a unique solution for $u$, and hence, a direct discretization would lead to a singular system. 

We illustrate this degeneracy in one dimension.  To this end, let $h$ be a fixed mesh size and $x_i=ih$ be equally distributed nodes. The range of $i$ is $-N,\dots,N$, with $N \in \N$ large enough such that the support of the density is within the computational domain $[-x_N, x_N]$. The time dependence is irrelevant for the discretization of  \eqref{eq:u-implicit-full} and we drop it in the calculations below.

Denote by $\rho_i$ and $u_i$, $i=-N,\dots,N$, the numerical approximations of $\rho(x_i)$ and $u(x_i)$, respectively. Also, denote by $\rhonum$ and $\unum$ the column vectors containing these values:
\[
 \rhonum = (\rho_{-N}, \dots, \rho_N)^T, \qquad \unum = (u_{-N}, \dots, u_N)^T.
\]
We choose the midpoint rule to approximate the integrals in \eqref{eq:u-implicit-full}; note however that the argument below can be adapted to apply to higher order quadrature rules. This gives
\begin{align*}
   \Phi(x_i)
   =\int_\R \phi(|x_i-y|)\rho(y)dy
   ~\approx~& h \sum_{k} \phi(|i-k|h) \rho_k,
\\
   \int_\R \phi(|x_i-y|) \rho(y) u(y) dy
   ~\approx~& h \sum_k \phi(|i-k|h) \rho_k u_k,
\\
\int_\R K'(x_i-y)\rho(y)dy
    ~\approx~& h \sum_k K' ((i-k)h) \rho_k.
\end{align*}

Introduce notations
\[
\phi_i=\phi(|i|h), \quad K'_i = K'(ih), \qquad i=-N,\dots,N.
\]
The following symmetries, respectively antisymmetries, hold:
\[
\phi_i = \phi_{-i}, \qquad K'_i = -K'_{-i}, \qquad \text{ for all } i=-N,\dots,N,
\]
where the antisymmetry of $K'_i$ follows from \eqref{eqn:Ksymmetric}. Also, since the influence function $\phi$ is assumed to be positive and non-increasing (in accord with hypotheses made throughout the paper), $\phi$ has a lower bound on the computational domain and we have $\phi_i \geq \eta$,  for all $i=-N,\dots,N$, where $\eta = \phi(2x_N)$. Note that we made an abuse of notation here, as a variable $\eta$, with a very similar meaning and role, was defined in \eqref{eqn:etadef}; however we try to parallel the analytical considerations made in the proof of Proposition \ref{Prop:2-1}, and  refrain from introducing unnecessary extra notations.

The discretization of \eqref{eq:u-implicit-full} reduces then to solving the linear system 
\begin{equation}
\label{eqn:linsys}
A \unum= \bnum,
\end{equation}
where
\begin{equation}
\label{eqn:matA}
\begin{aligned}
A=&\begin{bmatrix}
\sum_k \phi_{-N-k}\rho_k&&\\ &\ddots& \\ &&\sum_k\phi_{N-k}\rho_k
\end{bmatrix}-
\begin{bmatrix}
\phi_{-N-(-N)}\rho_{-N}&\cdots&\phi_{-N-N}\rho_N\\
 \vdots&\ddots&\vdots \\ \phi_{N-(-N)}\rho_{-N}&\cdots&\phi_{N-N}\rho_N
\end{bmatrix},
\end{aligned}
\end{equation}
and 
\begin{equation}
\label{eqn:b}
\bnum=(b_{-N},\dots,b_N)^T, \qquad b_i=-\sum_kK'_{i-k}\rho_k, \quad i=-N,\dots,N.
\end{equation}

The matrix $A$ is the discrete analogue of the operator $\CalA$ defined in \eqref{eq:calA}. Similar to the continuous case, matrix $A$ is singular, and solutions of the linear system are not unique. This fact is detailed in the following proposition.

\begin{Proposition}[Existence and non-uniqueness] \label{prop:matrix-A}
Matrix $A$ given by \eqref{eqn:matA} is singular with $dim(Null(A))=1$. Moreover, the discrete linear system \eqref{eqn:linsys}-\eqref{eqn:b} has infinitely many solutions. 
\end{Proposition}
\begin{proof}
It is easy to check that the entries of each row of $A$ add up to zero; therefore,
$A$ is singular. To show that  $dim(Null(A))=1$ we simply note that if we remove the $i$-th row
and the $i$-th column of $A$ (for an index $i$ such that $\rho_i \neq 0$), the remaining matrix is strictly diagonally
dominant, and hence it has full rank. 

To prove existence of infinitely many solutions, as opposed to no solution, we
need to check $\bnum \in Range(A)$. Indeed, by symmetry of $\phi_i$, $A^T \vec{\rho} =0$, and hence
\[
Null(A^T)=span\{\vec\rho\, \}.
\]
Finally, by the antisymmetry of $K'_i$, one has $\bnum \perp \rhonum$, which, by Fredholm alternative, yields the conclusion.
\end{proof}

Using Propostion~\ref{prop:matrix-A}, we immediately have
\begin{Proposition}[Uniqueness with momentum conservation]\label{prop:discreteunique} There exists a
  unique solution $\unum$ of \eqref{eqn:linsys}-\eqref{eqn:b}, subject to the (discrete) momentum conservation condition 
\begin{equation}
\label{eqn:discr-mom}
  \rhonum^{\;T} \unum=0.
\end{equation}
\end{Proposition}
\begin{proof}
Although this result can be derived directly from Proposition~\ref{prop:matrix-A}, using the structure of $Null(A)$ deduced in its proof, we show here a more constructive way, that is also more relevant to the numerical implementation. Let $\enum=(1,\cdots,1)^T$. 
Then, a solution $\unum$ of  \eqref{eqn:linsys}-\eqref{eqn:b} which satisfies the discrete momentum conservation 
condition \eqref{eqn:discr-mom}, also solves
\begin{equation}
\label{eqn:linsysM}
M \unum= \bnum,
\end{equation}
with
\begin{equation}
\label{eqn:matM} 
M=A+\eta \, \enum  \rhonum^{\; T}.
\end{equation}
Here, $\eta$ represents the positive lower bound of $\phi_i$, $i=-N,\dots,N$, mentioned above. 

The matrix $M$ has the form
\begin{align*}
M=&\begin{bmatrix}
\sum_k\phi_{-N-k}\rho_k&&\\ &\ddots& \\ &&\sum_k\phi_{N-k}\rho_k
\end{bmatrix}-
\begin{bmatrix}
(\phi_{-N-(-N)}-\eta)\rho_{-N}&\cdots&(\phi_{-N-N}-\eta)\rho_N\\
 \vdots&\ddots&\vdots \\ (\phi_{N-(-N)}-\eta)\rho_{-N}&\cdots&(\phi_{N-N}-\eta)\rho_N
\end{bmatrix}.
\end{align*}
Since all the off-diagonal entries of $M$ are negative, $M$ can be shown to be strictly diagonally
dominant. Hence, $M$ is an invertible matrix, and \eqref{eqn:linsysM} has a unique solution.
\end{proof}
\begin{rmk} The discrete theory aligns perfectly well with the results in Section \ref{sect:well-posedness}. In particular,
matrices $A$ and $M$ are the discrete analogues of the singular, respectively invertible, operators $\CalA$ and $\CalM$ defined in \eqref{eq:calA} and \eqref{eq:calM}, while the conservation \eqref{eqn:discr-mom} of the discrete linear momentum is used as an auxiliary condition that enforces uniqueness of the numerical solution. 
\end{rmk}

To conclude, our numerical discretization of \eqref{eq:u-implicit-full} consists in the
following procedure. First, form the matrix $A$ using appropriate
quadrature rules on integrations. Next, change the singular matrix $A$
into the invertible matrix $M$ given by \eqref{eqn:matM}, and solve the linear system \eqref{eqn:linsysM}. For an efficient numerical implementation, it is important to note that the velocities $u_i$ at grid points with  nontrivial density values $\rho_i \neq 0$ do {\em not} depend on values $u_j$ for which $\rho_j=0$;
this comment is in fact the discrete counterpart of Remark \ref{remark:outside}. Hence, one can consider only the nodes (or indices) that correspond to nontrivial densities and reduce the size of the system. 

\medskip
\subsection{Numerical examples} 
Below we illustrate our method with three numerical simulations: two examples in 1D with smooth and non-smooth potentials, respectively, and a 2D example with a smooth potential. We make choices of potentials for which explicit steady states 
of \eqref{eq:macroscopic-full} can be calculated.
The influence function in all the cases is chosen as
\begin{equation}
\label{eqn:phiex}
\phi(r)=(1+r^2)^{-1/2}.
\end{equation}

\smallskip
{\bf Example 1}  (1D with smooth potential). 
We consider a power-law interaction potential with quartic attraction and quadratic repulsion:
\begin{equation}
\label{eqn:Kex1}
K(x)=\frac{x^4}{4}-\frac{x^2}{2}.
\end{equation}
The derivative $K'$ is smooth, but not bounded, which means that $K'$ lies in ${W}_{loc}^{1,\infty}(\R)$, but not in ${W}^{1,\infty}(\R)$, as assumed by theoretical results in previous sections. In simulations however, we deal with a bounded computational domain, on which $K'$ and $K''$ are of course, bounded.

The initial condition is chosen to be a smooth bump function
\begin{equation}
\label{eqn:ic-bump}
\rho_0(x)=C_se^{-\frac{1}{s^2-x^2}}\One_{(-s,s)}(x),
\end{equation}
where $s$ is a positive integer which determines the support of $\rho_0$, $\One$ denotes the indicator function, and $C_s$ is a normalization constant such that
$\|\rho_0\|_{L^1}=1$.

Figure \ref{fig:Ex1} shows the time evolution of $\rho$ starting from initial data \eqref{eqn:ic-bump} with $s=0.5$ (top row) and $s=1.5$ (bottom row). The computational domain is set to be $[-2,2]$ and the mesh size is $h=.02$. In both cases we observe that the solution approaches (by expanding, respectively compressing) the same steady state $\rho_\infty$ given by two Dirac delta distributions located at $-1/2$ and $1/2$:
\begin{equation}
\label{eqn:ss-twodelta}
\rho_\infty(x)=\frac{1}{2}\left[\delta(x-1/2)+\delta(x+1/2)\right].
\end{equation}
We note that \eqref{eqn:ss-twodelta} is in fact a steady state of the attraction-repulsion model (with no alignment), i.e,
\begin{equation}\label{eq:att-rep}
\partial_t\rho-\Gradx \cdot (\rho(\Gradx K\ast\rho))=0.
\end{equation}
\begin{figure}[h]
\includegraphics[width=\textwidth]{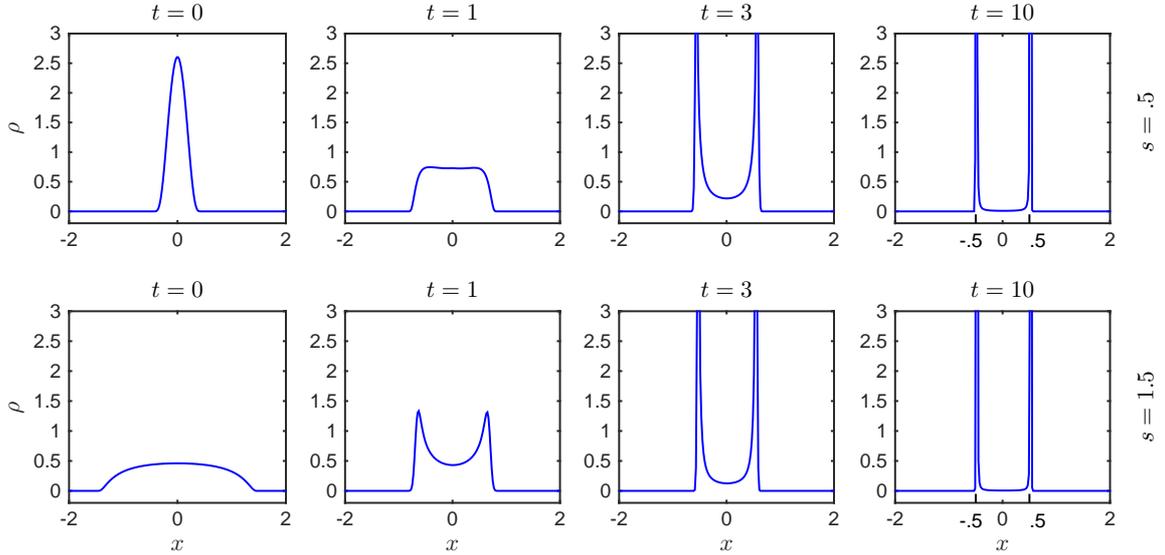}
\caption{Time evolution of the solution $\rho(t,x)$ of \eqref{eq:macroscopic-full} in one dimension, starting from the initial density \eqref{eqn:ic-bump} with  $s=0.5$ (top row) and
  $s=1.5$ (bottom row). The influence function is given by \eqref{eqn:phiex} and the interaction potential by \eqref{eqn:Kex1}. The two density profiles expand, respectively compress,  toward the same equilibrium state given by \eqref{eqn:ss-twodelta}.}
 \label{fig:Ex1}
\end{figure} 

\smallskip
{\bf Example 2} (1D with Morse potential).
We now take a Morse-type interaction potential:
\begin{equation}
\label{eqn:Kex2}
K(x)=-e^{-|x|/2}+e^{-|x|}.
\end{equation}
Note again that due to the singularity at the origin,  $K' \not \in W^{1,\infty}(\R)$.  In terms of numerical simulations however, singularities in the potential are not an issue of concern (at discrete level, a pointy potential and a very localized regularization of it, are essentially the same).

In this simulation we choose a more general non-smooth and non-symmetric initial density:
\begin{equation}
\label{eqn:ic-gen}
\rho_0(x)=\frac{1}{2}(1-x)\One_{(-1,1)}(x).
\end{equation}

Figure \ref{fig:Ex2} shows the time evolution of the solution (solid line) at  $t=0, 2, 5, 20$. In the same plot we show (dashed-line) the steady state of the attractive-repulsive equation \eqref{eq:att-rep} with the Morse-type
potential \eqref{eqn:Kex2}. An exact, explicit form of this equilibrium solution was derived in \cite{BT2011}:
\begin{equation}
\label{eqn:ss-Morse}
\rho_\infty(x)=C\left(\cos(\mu(x-c))-\lambda\right)\cdot\One_{|x-c|<H}(x),
\end{equation}
where the constants are given as follows: $\mu=1/\sqrt{2}$, $\lambda=-1/(3+\pi/\sqrt{2})$, $H=\pi/\sqrt{2}$. Also, $C=3/(2(\pi+3 \sqrt{2}))$ is a scaling that sets the mass to one, and $c=1/3$ is a shift of the centre of mass to the origin.
Similar to Example 1, we observe that the solution approaches asymptotically the steady state $\rho_\infty$.

\begin{figure}[h]
\includegraphics[width=\textwidth]{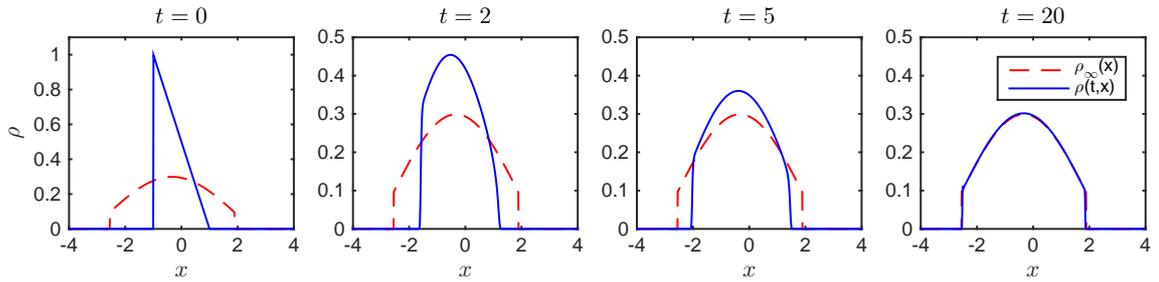}
\caption{Time evolution of the solution $\rho(t,x)$ of \eqref{eq:macroscopic-full} in one dimension, starting from the initial density \eqref{eqn:ic-gen}. The influence function is given by \eqref{eqn:phiex} and the interaction potential by \eqref{eqn:Kex2}. The solution (solid curve) approaches asymptotically the equilibrium state $\rho_\infty$ (dashed line) of the attractive-repulsive model \eqref{eq:att-rep}, given explicitly by \eqref{eqn:ss-Morse}.}
 \label{fig:Ex2}
\end{figure} 

\smallskip
{\bf Example 3} (2D with Newtonian-quadratic potential).
In the final example, we check the asymptotic behaviour of solutions to \eqref{eq:macroscopic-full}  in 2D.
The algorithm presented above can be immediately extended to two dimensions. 
We take an interaction potential $K$ with Newtonian repulsion and quadratic attraction:
\begin{equation}
\label{eqn:Kex3}
K(x)=-\log|x|+\frac{1}{2}x^2.
\end{equation}
With this choice of $K$, it has been shown that the steady state of the attractive-repulsive model \eqref{eq:att-rep} is
the indicator function of a disk:
\begin{equation}
\label{eqn:ss-2d}
\rho_\infty(x)=C\One_{|x|<1}(x-c),
\end{equation}
where the scale $C$ is set by conservation of mass, and $c$ is a shift of the centre of mass  \cite{FeHuKo11}.

We take the following radially symmetric initial density:
\begin{equation}
\label{eqn:ic-2d}
\rho_0(r) = (r-1)^2(r-2)^2\One_{1<r<2}(r).
\end{equation}
The numerical time evolution illustrated in Figure \ref{fig:ex3} indicates again, that  the solution approaches asymptotically $\rho_\infty$ (here $C=0.1$ and $c=0$). 
\begin{figure}[h]
\includegraphics[width=\textwidth]{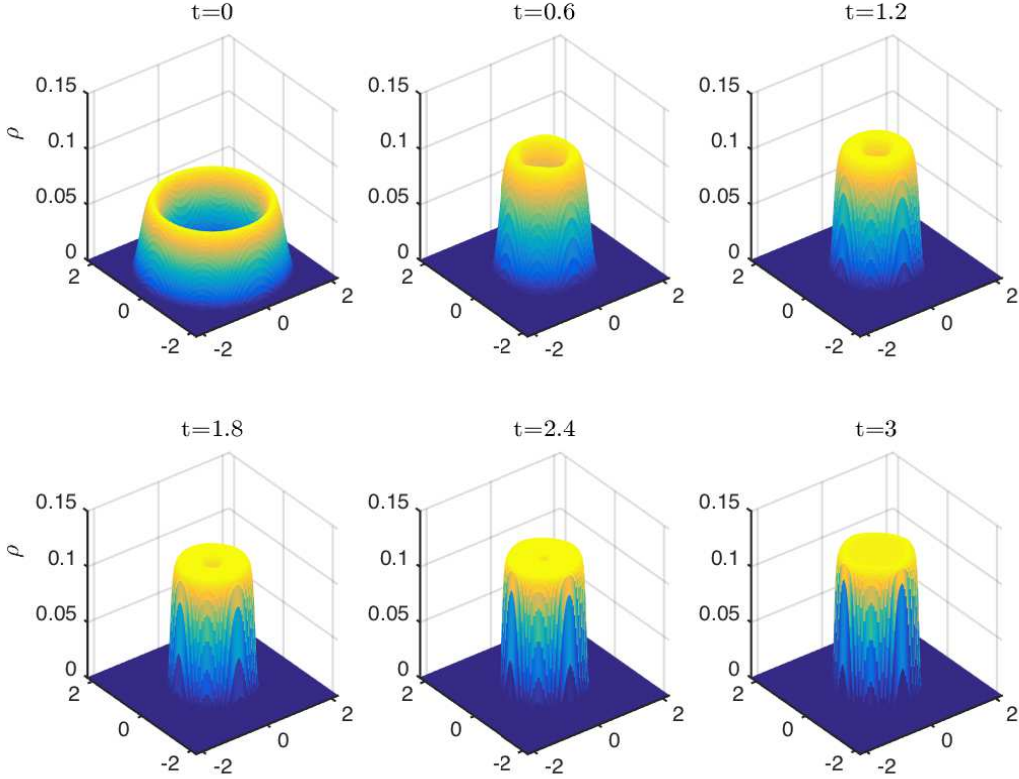}
\caption{Time evolution of the density $\rho(t,x)$ of the aggregation model \eqref{eq:macroscopic-full} in two dimensions, starting from the initial density \eqref{eqn:ic-2d}. The influence function is given by \eqref{eqn:phiex} and the interaction potential by \eqref{eqn:Kex3}. The solution approaches asymptotically the equilibrium state \eqref{eqn:ss-2d} consisting of a constant distribution in the unit disk.}
\label{fig:ex3}
\end{figure}

In this experiment we set the computational domain to be $[-2.2, 2.2]\times[-2.2,2.2]$ and fix an Eulerian computational grid with ${\rm d}x= {\rm d}y=0.04$.  We use a semi-Lagrangian scheme to evolve the density equation in time \cite{DoHuPoRo2004}. Specifically, we evolve numerically, using a forward Euler integrator with ${\rm d} t = 0.005$, the characteristic trajectories \eqref{eq:char-limit} and the density equation \eqref{eq:rho-full} written in characteristic form. After each time step, we interpolate the density values from the Lagrangian (distorted) grid back to the Eulerian grid, calculate the new velocity field according to the method described in Section \ref{subsect:discr-setting}, and repeat the procedure.

We conclude this section by noting that our preliminary numerical investigations of model \eqref{eq:macroscopic-full} (as presented here) suggest that the large time behaviour of solutions to \eqref{eq:macroscopic-full} is:
\[
\rho(t,x)\to\rho_\infty(x), \qquad \text{ for } x\in\Supp\rho(t), \quad \text{ as } t\to \infty,
\]
where $\rho_\infty$ is a steady state of \eqref{eq:att-rep}.
Consequently,
\[
u(t,x)\to0,\qquad \text { for all } x \in\Supp \rho (t), \quad \text { as } t \to \infty.
\]
This means that the macroscopic velocity $u$ converges to its (zero) average as time
goes to infinity, which corresponds to \emph{flocking}. Whether the model
always approaches asymptotically the flocking limit is left for future investigation.

\newpage
\bibliographystyle{elsarticle}
\bibliography{lit}

\begin{bibdiv}
\begin{biblist}

\bib{AGS2006}{book}{
      author={Ambrosio, Luigi},
      author={Gigli, Nicola},
      author={Savar{\'e}, Giuseppe},
       title={Gradient flows: in metric spaces and in the space of probability
  measures},
   publisher={Springer Science \& Business Media},
        date={2006},
}

\bib{AgIlRi2011}{article}{
      author={Agueh, Martial},
      author={Illner, Reinhard},
      author={Richardson, Ashlin},
       title={Analysis and simulations of a refined flocking and swarming model
  of {C}ucker-{S}male type},
        date={2011},
     journal={Kinet. Relat. Models},
      volume={4},
      number={1},
       pages={1\ndash 16},
}

\bib{BertozziCarilloLaurent}{article}{
      author={Bertozzi, Andrea~L.},
      author={Carrillo, Jos{\'e}~A.},
      author={Laurent, Thomas},
       title={Blow-up in multidimensional aggregation equations with mildly
  singular interaction kernels},
        date={2009},
     journal={Nonlinearity},
      volume={22},
      number={3},
       pages={683\ndash 710},
}

\bib{Burger:DiFrancesco}{article}{
      author={Burger, Martin},
      author={Di~Francesco, Marco},
       title={Large time behavior of nonlocal aggregation models with nonlinear
  diffusion},
        date={2008},
     journal={Netw. Heterog. Media},
      volume={3},
      number={4},
       pages={749\ndash 785},
}

\bib{Patrick1971}{book}{
      author={Billingsley, Patrick},
       title={Weak convergence of measures: Applications in probability},
   publisher={Society for Industrial and Applied Mathematics},
organization={CBMS-NSF Regional Conference Series in Applied Mathematics},
     address={Philadelphia, Pennsylvania},
        date={1971},
}

\bib{BertozziLaurent}{article}{
      author={Bertozzi, Andrea~L.},
      author={Laurent, Thomas},
       title={Finite-time blow-up of solutions of an aggregation equation in
  {$\mathbf{R}^n$}},
        date={2007},
     journal={Comm. Math. Phys.},
      volume={274},
      number={3},
       pages={717\ndash 735},
}

\bib{BertozziLaurentRosado}{article}{
      author={Bertozzi, Andrea~L.},
      author={Laurent, Thomas},
      author={Rosado, Jesus},
       title={${L}^p$ theory for the multidimensional aggregation equation},
        date={2011},
     journal={Comm. Pur. Appl. Math.},
      volume={64},
      number={1},
       pages={45\ndash 83},
}

\bib{BT2011}{article}{
      author={Bernoff, Andrew~J},
      author={Topaz, Chad~M},
       title={A primer of swarm equilibria},
        date={2011},
     journal={SIAM Journal on Applied Dynamical Systems},
      volume={10},
      number={1},
       pages={212\ndash 250},
}

\bib{BodnarVelasquez2}{article}{
      author={Bodnar, M.},
      author={Velazquez, J. J.~L.},
       title={An integro-differential equation arising as a limit of individual
  cell-based models},
        date={2006},
        ISSN={0022-0396},
     journal={J. Differential Equations},
      volume={222},
      number={2},
       pages={341\ndash 380},
}

\bib{CarrilloChoiHaurayCISM}{incollection}{
      author={Carrillo, J.A.},
      author={Choi, Y.P.},
      author={Hauray, M.},
       title={The derivation of swarming models: {M}ean-field limit and
  {W}asserstein distances},
        date={2014},
   booktitle={Collective dynamics from bacteria to crowds},
      series={CISM International Centre for Mechanical Sciences: Courses and
  Lectures},
      volume={553},
   publisher={Springer Vienna},
       pages={1\ndash 46},
}

\bib{CCH2015}{article}{
      author={Carrillo, Jos{\'e}~A},
      author={Chertock, Alina},
      author={Huang, Yanghong},
       title={A finite-volume method for nonlinear nonlocal equations with a
  gradient flow structure},
        date={2015},
     journal={Communications in Computational Physics},
      volume={17},
      number={01},
       pages={233\ndash 258},
}

\bib{CCR2011}{article}{
      author={Canizo, Jos{\'e}~A},
      author={Carrillo, Jos{\'e}~A},
      author={Rosado, Jes{\'u}s},
       title={A well-posedness theory in measures for some kinetic models of
  collective motion},
        date={2011},
     journal={Mathematical Models and Methods in Applied Sciences},
      volume={21},
      number={03},
       pages={515\ndash 539},
}

\bib{CD2010}{article}{
      author={Cucker, Felipe},
      author={Dong, Jiu-Gang},
       title={Avoiding collisions in flocks},
        date={2010},
     journal={IEEE Transactions on Automatic Control},
      volume={55},
      number={5},
       pages={1238\ndash 1243},
}

\bib{Camazine_etal}{book}{
      author={Camazine, Scott},
      author={Deneubourg, Jean-Louis},
      author={Franks, Nigel~R.},
      author={Sneyd, James},
      author={Theraulaz, Guy},
      author={Bonabeau, Eric},
       title={Self-organization in biological systems},
      series={Princeton Studies in Complexity},
   publisher={Princeton University Press},
     address={Princeton, NJ},
        date={2003},
        note={Reprint of the 2001 original},
}

\bib{CaFrFiLaSl2011}{article}{
      author={Carrillo, J.~A.},
      author={DiFrancesco, M.},
      author={Figalli, A.},
      author={Laurent, T.},
      author={Slep{\v{c}}ev, D.},
       title={Global-in-time weak measure solutions and finite-time aggregation
  for nonlocal interaction equations},
        date={2011},
     journal={Duke Math. J.},
      volume={156},
      number={2},
       pages={229\ndash 271},
}

\bib{Chuang_etal}{article}{
      author={Chuang, Yao-Li},
      author={D'Orsogna, Maria~R.},
      author={Marthaler, Daniel},
      author={Bertozzi, Andrea~L.},
      author={Chayes, Lincoln~S.},
       title={State transitions and the continuum limit for a 2{D} interacting,
  self-propelled particle system},
        date={2007},
     journal={Phys. D},
      volume={232},
      number={1},
       pages={33\ndash 47},
}

\bib{CFRT2010}{article}{
      author={Carrillo, Jos{\'e}~A},
      author={Fornasier, Massimo},
      author={Rosado, Jes{\'u}s},
      author={Toscani, Giuseppe},
       title={Asymptotic flocking dynamics for the kinetic {C}ucker-{S}male
  model},
        date={2010},
     journal={SIAM Journal on Mathematical Analysis},
      volume={42},
      number={1},
       pages={218\ndash 236},
}

\bib{CarrilloVecil2010}{incollection}{
      author={Carrillo, Jos{\'e}~A.},
      author={Fornasier, Massimo},
      author={Toscani, Giuseppe},
      author={Vecil, Francesco},
       title={Particle, kinetic, and hydrodynamic models of swarming},
        date={2010},
   booktitle={Mathematical modeling of collective behavior in socio-economic
  and life sciences},
      series={Model. Simul. Sci. Eng. Technol.},
   publisher={Birkh\"auser Boston, Inc., Boston, MA},
       pages={297\ndash 336},
}

\bib{Couzin_etal}{article}{
      author={Couzin, I.~D.},
      author={Krause, J.},
      author={James, R.},
      author={Ruxton, G.D.},
      author={Franks, N.~R.},
       title={Collective memory and spatial sorting in animal groups},
        date={2002},
     journal={J. Theor. Biol.},
      volume={218},
       pages={1\ndash 11},
}

\bib{CS2007}{article}{
      author={Cucker, Felipe},
      author={Smale, Steve},
       title={Emergent behavior in flocks},
        date={2007},
     journal={IEEE Transactions on Automatic Control},
      volume={52},
      number={5},
       pages={852\ndash 862},
}

\bib{D'Orsogna_etal}{article}{
      author={D'Orsogna, Maria~R.},
      author={Chuang, Yao-Li},
      author={Bertozzi, Andrea~L.},
      author={Chayes, Lincoln~S.},
       title={Self-propelled particles with soft-core interactions: patterns,
  stability and collapse},
        date={2006},
     journal={Phys. Rev. Lett.},
      volume={96},
      number={10},
       pages={104302},
}

\bib{DoHuPoRo2004}{incollection}{
      author={Donea, J.},
      author={Huerta, A.},
      author={Ponthot, J.-P.},
      author={Rodriguez-Ferran, A.},
       title={Arbitrary {L}agrangian-{E}ulerian {M}ethos},
        date={2004},
   booktitle={Encyclopedia of {C}omputational {M}echanics},
      editor={Stein, E.},
      editor={de~Borst, R.},
      editor={Hughes, T.J.R.},
      volume={1},
   publisher={John Wiley},
       pages={413\ndash 437},
}

\bib{FeHuKo11}{article}{
      author={Fetecau, R.~C.},
      author={Huang, Y.},
      author={Kolokolnikov, T.},
       title={Swarm dynamics and equilibria for a nonlocal aggregation model},
        date={2011},
     journal={Nonlinearity},
      volume={24},
      number={10},
       pages={2681\ndash 2716},
}

\bib{FS2014}{article}{
      author={Fetecau, Razvan},
      author={Sun, Weiran},
       title={First-order aggregation models and zero inertia limits},
        date={2014},
     journal={arXiv preprint arXiv:1410.7095},
}

\bib{HL2009}{article}{
      author={Ha, Seung-Yeal},
      author={Liu, Jian-Guo},
       title={A simple proof of the cucker-smale flocking dynamics and
  mean-field limit},
        date={2009},
     journal={Communications in Mathematical Sciences},
      volume={7},
      number={2},
       pages={297\ndash 325},
}

\bib{HT2008}{article}{
      author={Ha, Seung-Yeal},
      author={Tadmor, Eitan},
       title={From particle to kinetic and hydrodynamic descriptions of
  flocking},
        date={2008},
     journal={Kinetic and Related Models},
      volume={1},
      number={3},
       pages={415\ndash 435},
}

\bib{Jabin2000}{article}{
      author={Jabin, Pierre-Emmanuel},
       title={Macroscopic limit of {V}lasov type equations with friction},
        date={2000},
     journal={Annales de l'Institut Henri Poincare (C) Non Linear Analysis},
      volume={17},
      number={5},
       pages={651\ndash 672},
}

\bib{Jackson2010}{book}{
      author={Jackson, M.~O.},
       title={Social and {E}conomic {N}etworks},
   publisher={Princeton University Press},
        date={2010},
}

\bib{JiEgerstedt2007}{article}{
      author={Ji, M.},
      author={Egerstedt, M.},
       title={Distributed coordination control of multi-agent systems while
  preserving connectedness},
        date={2007},
     journal={IEEE Trans. Robot.},
      volume={23},
      number={4},
       pages={693\ndash 703},
}

\bib{KoSuUmBe2011}{article}{
      author={Kolokolnikov, Theodore},
      author={Sun, Hui},
      author={Uminsky, David},
      author={Bertozzi, Andrea~L.},
       title={A theory of complex patterns arising from 2{D} particle
  interactions},
        date={2011},
     journal={Phys. Rev. E, Rapid Communications},
      volume={84},
       pages={015203(R)},
}

\bib{LiLuKe2008}{article}{
      author={Li, Yue-Xian},
      author={Lukeman, Ryan},
      author={Edelstein-Keshet, Leah},
       title={Minimal mechanisms for school formation in self-propelled
  particles},
        date={2008},
     journal={Phys. D},
      volume={237},
      number={5},
       pages={699\ndash 720},
}

\bib{LeToBe2009}{article}{
      author={Leverentz, Andrew~J.},
      author={Topaz, Chad~M.},
      author={Bernoff, Andrew~J.},
       title={Asymptotic dynamics of attractive-repulsive swarms},
        date={2009},
     journal={SIAM J. Appl. Dyn. Syst.},
      volume={8},
      number={3},
       pages={880\ndash 908},
}

\bib{M&K}{article}{
      author={Mogilner, A.},
      author={Edelstein-Keshet, L.},
       title={A non-local model for a swarm},
        date={1999},
     journal={J. Math. Biol.},
      volume={38},
       pages={534\ndash 570},
}

\bib{Miller_etal2012}{article}{
      author={Miller, Jennifer~M.},
      author={Kolpas, Allison},
      author={Juchem~Neto, Joao~Plinio},
      author={Rossi, Louis~F.},
       title={A continuum three-zone model for swarms},
        date={2012},
     journal={Bull. Math. Biol.},
      volume={74},
      number={3},
       pages={536\ndash 561},
}

\bib{MT2011}{article}{
      author={Motsch, Sebastien},
      author={Tadmor, Eitan},
       title={A new model for self-organized dynamics and its flocking
  behavior},
        date={2011},
     journal={Journal of Statistical Physics},
      volume={144},
      number={5},
       pages={923\ndash 947},
}

\bib{MotschTadmor2014}{article}{
      author={Motsch, Sebastien},
      author={Tadmor, Eitan},
       title={Heterophilious dynamics enhances consensus},
        date={2014},
     journal={S{IAM} Review},
      volume={56},
      number={4},
       pages={577\ndash 621},
}

\bib{Reynolds1987}{article}{
      author={Reynolds, Craig~W},
       title={Flocks, herds and schools: A distributed behavioral model},
        date={1987},
     journal={ACM Siggraph Computer Graphics},
      volume={21},
      number={4},
       pages={25\ndash 34},
}

\bib{SR2014}{article}{
      author={Saint-Raymond, L.},
       title={A mathematical {PDE} perspective on the {C}hapman-{E}nskog
  expansion},
        date={2014},
     journal={Bull. Amer. Math. Soc.},
      volume={51},
       pages={247\ndash 275},
}

\bib{Tan2014}{article}{
      author={Tan, Changhui},
       title={A discontinuous {G}alerkin method on kinetic flocking models},
        date={2014},
     journal={arXiv preprint arXiv:1409.5509},
}

\bib{Topaz:Bertozzi}{article}{
      author={Topaz, C.~M.},
      author={Bertozzi, A.~L.},
       title={Swarming patterns in a two-dimensional kinematic model for
  biological groups},
        date={2004},
     journal={SIAM J. Appl. Math.},
      volume={65},
       pages={152\ndash 174},
}

\bib{Tikhonov1952}{article}{
      author={Tikhonov, A.~N.},
       title={Systems of differential equations containing small parameters in
  the derivatives},
        date={1952},
     journal={Mat. Sb. (N.S.)},
      volume={31(73)},
       pages={575\ndash 586},
}

\bib{TT2014}{article}{
      author={Tadmor, Eitan},
      author={Tan, Changhui},
       title={Critical thresholds in flocking hydrodynamics with non-local
  alignment},
        date={2014},
     journal={Philosophical Transactions of the Royal Society A: Mathematical,
  Physical and Engineering Sciences},
      volume={372},
      number={2028},
       pages={20130401},
}

\bib{Vasileva1963}{article}{
      author={Vasil{\cprime}eva, A.~B.},
       title={Asymptotic behaviour of solutions of certain problems for
  ordinary non-linear differential equations with a small parameter multiplying
  the highest derivatives},
        date={1963},
     journal={Uspehi Mat. Nauk},
      volume={18},
      number={3 (111)},
       pages={15\ndash 86},
}

\end{biblist}
\end{bibdiv}


\end{document}